\documentclass[reqno]{amsart}

\usepackage{amssymb}
\usepackage{amsfonts}
\usepackage{amsmath}
\usepackage{color}
\usepackage{amsaddr}

\newtheorem{theorem}{Theorem}[section]

\newtheorem{corollary}[theorem]{Corollary}

\newtheorem{definition}[theorem]{Definition}

\newtheorem{lemma}[theorem]{Lemma}

\newtheorem{proposition}[theorem]{Proposition}

\newtheorem*{problemp}{Problem P$_{\varepsilon}$}
\newtheorem*{problem0}{Problem P$_0$}

\theoremstyle{remark}

\newtheorem{remark}[theorem]{Remark}

\numberwithin{equation}{section}

\newcommand{\diff}{{\rm{d}}}
\newcommand{\ep}{\varepsilon}

\begin{document}

\title[Boundary Conditions with Memory]{Robust Exponential Attractors for Coleman--Gurtin Equations with Dynamic Boundary Conditions Possessing Memory}

\author[J. L. Shomberg]{Joseph L. Shomberg}



\address{Department of Mathematics and Computer Science, Providence
College, Providence, RI 02918, USA, \\
\tt{{jshomber@providence.edu} }}

\date{\today}

\begin{abstract}
The well-posedness of a generalized Coleman--Gurtin equation equipped with dynamic boundary conditions with memory was recently established by the author with C.G. Gal. 
In this article we report advances concerning the asymptotic behavior and stability of this heat transfer model. 
For the model under consideration, we obtain a family of exponential attractors that is robust/H\"{o}lder continuous with respect to a perturbation parameter occurring in a singularly perturbed memory kernel. 
We show that the basin of attraction of these exponential attractors is the entire phase space.
The existence of (finite dimensional) global attractors follows.
The results are obtained by assuming the nonlinear terms defined on the interior of the domain and on the boundary satisfy standard dissipation assumptions. 
Also, we work under a crucial assumption that dictates the memory response in the interior of the domain matches that on the boundary.
\end{abstract}

\maketitle

\tableofcontents

%===========================================================
\section{Introduction to the model problem}
%===========================================================

In the framework of \cite{GPM98}, let us only consider a thermodynamic
process based on heat conduction. 
Suppose that a bounded domain $\Omega
\subset \mathbb{R}^{n},n\geq 1,$ is occupied by a body which may be
inhomogeneous, but has a configuration constant in time. 
Thermodynamic
processes taking place inside $\Omega ,$ with sources also present at the
boundary $\Gamma ,$ give rise to the following model for the temperature
field $u$:
\begin{align}
& \partial_{t}u - \omega\Delta u - (1-\omega) \int_{0}^{\infty} k(s) \Delta u(x,t-s) \diff s + f(u)  \label{eq1m} \\ 
& + \alpha(1-\omega) \int_0^\infty k(s)u(x,t-s) \diff s =0,  \notag
\end{align}
in $\Omega \times \left( 0,\infty \right)$, subject to the following boundary condition:
\begin{align}
& \partial_{t}u - \omega\Delta _{\Gamma }u + \omega\partial_{\mathbf{n}}u + (1-\omega) \int_0^\infty k(s)\partial_{\mathbf{n}} u(x,t-s)\diff s \label{eq2m} \\ 
& + (1-\omega) \int_0^\infty k(s)(-\Delta_\Gamma+\beta)u (x,t-s)\diff s + g(u) = 0,  \notag 
\end{align}
on $\Gamma \times (0,\infty),$ for every $\alpha\ge0$, $\beta\ge0$, $\omega \in \lbrack
0,1)$, and where $k:[0,\infty )\rightarrow \mathbb{R}$ is a continuous nonnegative function, smooth on $(0,\infty)$, vanishing at infinity and
satisfying the relation
\begin{equation*}
\int_{0}^{\infty} k(s) \diff s = 1,
\end{equation*}
$\partial_{\mathbf{n}}$ represents the normal derivative and $-\Delta_\Gamma$ is the Laplace-Beltrami operator.
The cases $\omega =0$ and $\omega >0$ in (\ref{eq1m}) are usually referred
as the Gurtin--Pipkin and the Coleman--Gurtin models, respectively. 
The literature contains a full treatment of equation (\ref{eq1m}) only in the
case of standard boundary conditions (Dirichlet, Neumann and periodic boundary conditions).
In light of new results and extensions for the phase field equations (see, e.g., \cite{CGGM10, GGM08} and references therein), we must consider more general dynamic boundary conditions. 
In particular, we quote \cite{GMS2010}:
\begin{quote}
In most works, the equations are endowed with Neumann boundary conditions for both [unknowns] $u$ and $w$ (which means that the interface is orthogonal to the boundary and that there is no mass flux at the boundary) or with periodic boundary conditions. 
Now, recently, physicists have introduced the so-called dynamic boundary conditions, in the sense that the kinetics, i.e., $\partial_t u$, appears explicitly in the boundary conditions, in order to account for the interaction of the components with the walls for a confined system.
\end{quote}
The derivation of (\ref{eq2m}) in the context of (\ref{eq1m}) can be derived in a similar fashion as in \cite{Gal&Grasselli08, Gold06} exploiting first and second laws of thermodynamics. 
Let $\omega \in \lbrack 0,1)$ be fixed. It is clear that if we (formally) choose $k=\delta _{0}$ (the Dirac mass at zero), equations (\ref{eq1m})-(\ref{eq2m}) turn into the following system:
\begin{equation}
\partial_{t}u - \Delta u + f(u) + \alpha(1-\omega)u = 0, \quad \text{in}\ \Omega \times (0,\infty),  \label{eq3d}
\end{equation}
\begin{equation}
\partial_{t}u - \Delta_{\Gamma}u + \partial_{\mathbf{n}}u + g(u) + \beta(1-\omega) u = 0, \quad \text{on}\ \Gamma \times \left( 0,\infty \right).
\label{eq4d}
\end{equation}
The latter has been investigated quite extensively recently in many contexts
(i.e., phase-field systems, heat conduction with a source at $\Gamma $, Stefan problems, etc).

Now we define, for $\varepsilon \in (0,1]$, 
\begin{equation*}
k_{\varepsilon}(s) = \frac{1}{\varepsilon}k\left(\frac{s}{\varepsilon}\right),
\end{equation*}
and we consider the same family of equations (\ref{eq1m})-(\ref{eq2m}), replacing $k$ with $k_{\varepsilon }$. 
Thus, $k_{\varepsilon }\rightarrow\delta _{0}$ when $\varepsilon \rightarrow 0$.
Our goal is to show in what sense does the system (\ref{eq1m})-(\ref{eq2m}) converge to (\ref{eq3d})-(\ref{eq4d}) as $\varepsilon \rightarrow 0.$ 

Such results seem to have begun with the hyperbolic relaxation of a Chaffee--Infante reaction diffusion equation in \cite{Hale&Raugel88}.
The motivation for such a hyperbolic relaxation is similar to the motivation for applying a memory relaxation; it alleviates the parabolic problems from the sometimes unwanted property of ``infinite speed of propagation''.
In \cite{Hale&Raugel88} however, Hale and Raugel proved the existence of a family of global attractors that is upper-semicontinuous in the phase space. 
A global attractor is a unique compact invariant subset of the phase space that attracts all trajectories of the associated dynamical system, even at arbitrarily slow rates (cf. \cite{Kostin98} and \cite[Theorem 14.6]{Robinson01}).
In a sense which will become clearer below, upper-semicontinuity guarantees the attractors to not ``blow-up'' as the perturbation parameter vanishes; i.e.,
\[
\sup_{x\in A_\ep}\inf_{y\in A_0}\|x-y\|_{X_\ep}\longrightarrow 0 \quad\text{as}\quad \ep\rightarrow 0^+.
\]

Unlike global attractors, exponential attractors (sometimes called, inertial sets) are compact positively invariant sets possessing finite fractal dimension that attract bounded subsets of the phase space exponentially fast. 
It can readily be seen that when both a global attractor $\mathcal{A}$ and an exponential attractor $\mathfrak{M}$ exist, then $\mathcal{A}\subseteq \mathfrak{M}$ provided that the basin of attraction of $\mathfrak{M}$ is the whole phase space, and so the global attractor is also finite dimensional. 
When we turn our attention to proving the existence of exponential attractors, certain higher-order dissipative estimates are required. 
In some interesting cases, it has not yet been shown how to obtain the appropriate estimates (which would provide the existence of a compact absorbing set, for example) {\em{independent}} of the perturbation parameter (cf. e.g. \cite{Frigeri&ShombergXX,Gal&Shomberg15}).
It is precisely because we are able to provide a higher-order uniform bound for the model problems here that we do not give a separate upper-semicontinuity result for the global attractors.
An appropriate uniform higher-order bound will essentially/almost mean that a robustness result may be found (but it is not guaranteed).

Robust families of exponential attractors (that is, both upper- and lower-semicontinuous with explicit control over semidistances in terms of the perturbation parameter) of the type reported in \cite{GGMP05} have successfully been shown to exist in many different applications, of which we will limit ourselves to mention only \cite{GMPZ10} which contains some applications of memory relaxation of reaction diffusion equations: Cahn--Hilliard equations, phase-field equations, wave equations, beam equations, and numerous others. 
The main idea behind robustness is typically an estimate of the form, 
\begin{equation}  \label{robust-intro}
\|S_\varepsilon(t)x-\mathcal{L}S_0(t)\Pi x\|_{X_\varepsilon}\leq C\varepsilon^p,
\end{equation}
for all $t$ in some interval, where $x\in X_\varepsilon$, $S_\varepsilon(t):X_\ep\rightarrow X_\ep$ and $S_0(t):X_0\rightarrow X_0$ are semigroups generated by the solutions of the perturbed problem and the limit problem, respectively, $\Pi$ denotes a projection from $X_\varepsilon$ onto $X_0$ and $\mathcal{L}$ is a ``lift'' from $X_0$ into $X_\varepsilon$, and finally $C,p>0$ are constants.
Controlling this difference in a suitable norm is crucial to obtaining our continuity results (see (C5) in Proposition \ref{abstract2}). 
The estimate (\ref{robust-intro}) means we can approximate the limit problem with the perturbation with control explicitly written in terms of the perturbation parameter. 
Usually such control is only exhibited on compact time intervals. 
Observe, a result of this type will ensure that for every problem of type (\ref{eq3d})-(\ref{eq4d}), there is an ``memory relaxation'' of the form (\ref{eq1m})-(\ref{eq2m}) {\em{close by}} in the sense that the difference of corresponding trajectories satisfies \eqref{robust-intro}.

We carefully treat the following issues:

\begin{description}
\item[1] Well-posedness of the system comprising of equations (\ref{eq1m})-(\ref{eq2m}) and \eqref{eq3d}-\eqref{eq4d}.

\item[2] Dissipation: the existence of bounded absorbing set, and a {\em{compact}} absorbing set, each of which is uniform with respect to the perturbation parameter $\ep$.

\item[3] Stability: existence of a family of exponential attractors for each $\ep\in[0,1]$ and an analysis of the continuity properties (robustness/H\"{o}lder) with respect to $\ep$.

\item[4] The basin of attraction for each exponential attractor is the entire phase space, and in demonstrating this result we see that the semigroup of solution operators also admits a family of global attractors.
\end{description}

Concerning Issue 1, the well-posedness for a more general system, which includes the one above, was given recently by \cite{Gal-Shomberg15-2}. 
The relevant results from that work are cited below in Section 2. 
In this article we explore Issues 2, 3, and 4 in much more depth; in particular, the existence of an exponential attractor for each $\varepsilon\in[0,1]$, and the continuity of these attractors with respect to $\varepsilon$.

As is now customary (cf. \cite{CDGP-2010,CPS05,CPS06,Grasselli&Pata02-2}) we introduce the so-called integrated past history of $u$, i.e., the auxiliary variable
\begin{equation*}
\eta ^{t}(x,s) =\int_{0}^{s} u(x,t-y) \diff y,
\end{equation*}
for $s,t>0.$ 
Setting 
\begin{equation*}  
\mu(s) = -(1-\omega)k'(s),
\end{equation*}
formal integration by parts into (\ref{eq1m})-(\ref{eq2m}) yields
\begin{align}
(1-\omega) \int_{0}^{\infty} k_{\varepsilon}(s) \Delta u(x,t-s) \diff s & =\int_{0}^{\infty}\mu_{\varepsilon}(s) \Delta \eta^{t}(x,s) \diff s, \notag \\ 
(1-\omega) \int_{0}^{\infty} k_{\varepsilon}(s) u(x,t-s) \diff s & = \int_{0}^{\infty} \mu_{\varepsilon}(s) \eta^{t}(x,s) \diff s, \notag \\ 
(1-\omega) \int_{0}^{\infty } k_{\varepsilon}(s) \partial_{\mathbf{n}} u(x,t-s) \diff s & = \int_{0}^{\infty} \mu_{\varepsilon}(s) \partial_{\mathbf{n}} \eta^{t}(x,s) \diff s, \notag 
\end{align}
and
\begin{equation*}
(1-\omega) \int_{0}^{\infty} k_{\varepsilon}(s)\left(-\Delta_{\Gamma} + \beta\right) u(x,t-s) \diff s = \int_{0}^{\infty} \mu_{\varepsilon}(s) \left(-\Delta_{\Gamma} + \beta\right) \eta^t(x,s) \diff s. 
\end{equation*}
where 
\begin{equation}  \label{mu-scaled}
\mu_{\varepsilon }(s) = \frac{1}{\varepsilon^{2}}\mu\left( \frac{s}{\varepsilon}\right).
\end{equation}

For each $\ep\in(0,1],$ the (perturbation) problem under consideration can now be stated.

\begin{problemp}
Let $\alpha,\beta\ge0,$ and $\omega\in(0,1)$. 
Find a function $(u,\eta)$ such that 
\begin{equation}
\partial_tu-\omega\Delta u-\int_0^\infty\mu_\ep(s)\Delta\eta^t(s)\diff s+\alpha\int_0^\infty\mu_\ep(s)\eta^t(s)\diff s+f(u)=0  \label{problemp-1}
\end{equation}
in $\Omega\times(0,\infty),$ subject to the boundary conditions
\begin{equation}
\partial_tu-\omega\Delta_\Gamma u+\omega\partial_{\mathbf{n}}u+\int_0^\infty\mu_\ep(s)\partial_{\mathbf{n}}\eta^t(s)\diff s+\int_0^\infty\mu_\ep(s)\left( -\Delta_\Gamma+\beta \right)\eta^t(s)\diff s+g(u)=0  \label{problemp-2}
\end{equation}
on $\Gamma\times(0,\infty),$ and 
\begin{equation}
\partial_t\eta^t(s)+\partial_s\eta^t(s)=u(t)\quad \text{in}\quad {\overline{\Omega}}\times(0,\infty),  \label{problemp-3}
\end{equation}
with
\begin{equation}
\eta^t(0)=0\quad \text{in}\quad {\overline{\Omega}}\times(0,\infty),  \label{problemp-4}
\end{equation}
and the initial conditions 
\begin{equation}
u(0)=u_0\quad \text{in}\quad \Omega, \quad u(0)=v_0\quad \text{on}\quad \Gamma,  \label{problemp-5}
\end{equation}
\begin{equation}
\eta^0(s)=\eta_0:=\int_0^su_0(x,-y)\diff y \quad\text{in}\quad \Omega, \quad \text{for}\quad s>0,  \label{problemp-6}
\end{equation}
and
\begin{equation}
\eta^0(s)=\xi_0:=\int_0^sv_0(x,-y)\diff y \quad\text{on} \quad \Gamma, \quad \text{for}\quad s>0.  \label{problemp-7}
\end{equation}
\end{problemp}

We will also discuss the problem corresponding to $\ep=0$.
The results for this problem may already be found in works in parabolic equations and the Wentzell Laplacian (see \cite{Gal12-2,Gal12-1,Gal-15Z,Gal&Warma10}).
The singular (limit) problem is

\begin{problem0}
Let $\alpha,\beta\ge0$ and $\omega\in(0,1)$. 
Find a function $u$ such that 
\begin{equation}
\partial_{t}u - \Delta u + f(u) + \alpha(1-\omega)u = 0  \label{problem0-1}
\end{equation}
in $\Omega\times(0,\infty),$ subject to the boundary conditions
\begin{equation}
\partial_{t}u - \Delta_{\Gamma}u + \partial_{\mathbf{n}}u + g(u) + \beta(1-\omega) u = 0  \label{problem0-2}
\end{equation}
on $\Gamma\times(0,\infty),$ with the initial conditions 
\begin{equation}
u(0)=u_0\quad \text{in}\quad \Omega \quad\text{and}\quad u(0)=v_0\quad \text{on}\quad \Gamma.  \label{problem0-5}
\end{equation}
\end{problem0}

\begin{remark}  \label{on-traces}
It need not be the case that the boundary traces of $u_0$ and $\eta_0$ be equal to $v_0$ and $\xi_0$, respectively. 
Thus, we are solving a much more general problem in which equation (\ref{problemp-1}) is interpreted as an evolution equation in the bulk $\Omega$ properly coupled with the equation (\ref{problemp-2}) on the boundary $\Gamma$.
Finally, from now on both $\eta_0$ and $\xi_0$ will be regarded as independent of the initial data $u_0$ and $v_0.$
Indeed, below we will consider a more general problem with respect to the original one. 
This will require a rigorous notion of solution to Problem {\textbf{P}}$_{\varepsilon}$ (cf. Definitions \ref{d:weak-solution}, \ref{d:strong-solution}), hence we introduce the functional setting associated with this system.
\end{remark}

Here below is the framework used to prove Hadamard well-posedness for Problem {\textbf{P}}$_{\varepsilon}$. 
Consider the space $\mathbb{X}^{2}:=L^{2}\left( \overline{\Omega }, \diff\mu \right) ,$ where
\begin{equation*}
\diff\mu = \diff x_{\mid \Omega }\oplus \diff\sigma ,
\end{equation*}%
where $\diff x$ denotes the Lebesgue measure on $\Omega $ and $\diff\sigma $ denotes the natural surface measure on $\Gamma $. It is
easy to see that $\mathbb{X}^{2}=L^{2}\left( \Omega , \diff x\right)
\oplus L^{2}\left( \Gamma , \diff\sigma \right) $ may be identified
under the natural norm
\begin{equation*}
\left\Vert u\right\Vert _{\mathbb{X}^{2}}^{2}=\int_{\Omega}\left\vert u\right\vert^{2} \diff x + \int_{\Gamma }\left\vert u\right\vert^{2} \diff\sigma.
\end{equation*}%
Moreover, if we identify every $u\in C(\overline{\Omega})$ with $U=\left( u_{\mid \Omega },u_{\mid \Gamma }\right) \in C\left( \Omega
\right) \times C\left( \Gamma \right) $, we may also define $\mathbb{X}^{2}$
to be the completion of $C(\overline{\Omega})$ in the norm $\|\cdot\|_{\mathbb{X}^{2}}$. 
In general, any function $u\in \mathbb{X}^{2}$ will be of the form $u=\binom{u_{1}}{u_{2}}$ with $u_{1}\in L^{2}\left( \Omega , \diff x\right) $ and $u_{2}\in L^{2}\left( \Gamma , \diff\sigma \right),$ and there need not be
any connection between $u_{1}$ and $u_{2}$. 
From now on, the inner product in the Hilbert space $\mathbb{X}^{2}$ will be denoted by $\left\langle \cdot,\cdot \right\rangle _{\mathbb{X}^{2}}.$ 
Hereafter, the spaces $L^{2}\left(\Omega , \diff x\right) $ and $L^{2}\left(\Gamma, \diff\sigma \right)$ will simply be denoted by $L^{2}\left(\Omega\right) $ and $L^{2}\left( \Gamma\right)$.

Recall that the Dirichlet trace map ${\mathrm{tr_{D}}}:C^{\infty }(\overline{\Omega}) \rightarrow C^{\infty}(\Gamma),$ defined by ${\mathrm{tr_{D}}}\left( u\right) =u_{\mid \Gamma }$ extends to a linear continuous operator ${\mathrm{tr_{D}}}:H^{r}\left( \Omega \right)
\rightarrow H^{r-1/2}\left( \Gamma \right) ,$ for all $r>1/2$, which is onto
for $1/2<r<3/2.$ This map also possesses a bounded right inverse ${\mathrm{tr_{D}}}^{-1}:H^{r-1/2}(\Gamma) \rightarrow H^{r}(\Omega)$ such that ${\mathrm{tr_{D}}}({\mathrm{tr_{D}}}^{-1}\psi) =\psi ,$ for any $\psi \in H^{r-1/2}(\Gamma) $. 
We can thus introduce the subspaces of $H^{r}(\Omega) \times H^{r}(\Gamma),$
\begin{equation}
\mathbb{V}^{r}:=\{\left( u,\psi \right) \in H^{r}\left( \Omega \right)
\times H^{r}\left( \Gamma \right) :{\mathrm{tr_{D}}}\left( u\right) =\psi \},
\label{vvv}
\end{equation}%
for every $r>1/2,$ and note that we have the following dense and compact
embeddings $\mathbb{V}^{r_{1}}\hookrightarrow \mathbb{V}^{r_{2}},$ for any $r_{1}>r_{2}>1/2$. Finally, we think of $\mathbb{V}^{1}\simeq H^{1}\left(\Omega \right) \oplus H^{1}\left( \Gamma \right) $ as the completion of $C^{1}\left( \overline{\Omega }\right) $ in the norm
\begin{equation}
\left\Vert u\right\Vert _{\mathbb{V}^{1}}^{2}:=\int_{\Omega } \left(
\left\vert \nabla u\right\vert ^{2} + \alpha\left\vert u\right\vert ^{2} \right) \diff x + \int_{\Gamma }\left(\left\vert \nabla _{\Gamma }u\right\vert
^{2}+\beta\left\vert u\right\vert^{2}\right) \diff\sigma  \label{v1b}
\end{equation}
(or some other equivalent norm in $H^{1}\left( \Omega \right) \times H^{1}\left( \Gamma \right) $).
Naturally, the norm on the space $\mathbb{V}^r$ is defined as
\begin{equation}\label{Vr-norm}
\|u\|^2_{\mathbb{V}^r} := \|u\|^2_{H^r(\Omega)} + \|u\|^2_{H^r(\Gamma)}.
\end{equation}

For $U=(u,u_{\mid\Gamma})^{\mathrm{tr}}\in\mathbb{V}^1$, let $C_\Omega>0$
denote the best constant in which the Sobolev-Poincar\'e inequality holds 
\begin{equation}  \label{Poincare}
\left\Vert u-\langle u \rangle_\Gamma \right\Vert _{L^{s}\left( \Omega \right) }
\leq C_\Omega\left\Vert \nabla u\right\Vert _{L^{s}(\Omega)},
\end{equation}
for $s\geq 1$ (see \cite[Lemma 3.1]{RBT01}).
Here $$\langle u \rangle_{\Gamma}:=\frac{1}{\left\vert \Gamma \right\vert }\int_{\Gamma}u_{\mid\Gamma} \diff\sigma.$$

Let us now introduce the spaces for the memory variable $\eta $. For a
nonnegative measurable function $\theta $ defined on $\mathbb{R}_{+}$ and a
real Hilbert space $W$ (with inner product denoted by $\left\langle \cdot
,\cdot \right\rangle_W$), let $L_{\theta }^{2}\left( \mathbb{R}%
_{+};W\right) $ be the Hilbert space of $W$-valued functions on $\mathbb{R}%
_{+}$, endowed with the following inner product%
\begin{equation}  
\left\langle \phi _{1},\phi _{2}\right\rangle _{L_{\theta}^{2}\left( 
\mathbb{R}_{+};W\right) }:=\int_{0}^{\infty }\theta(s)\left\langle \phi
_{1}\left( s\right) ,\phi _{2}\left( s\right) \right\rangle _W \diff s.
\end{equation}  \label{sc-2}
Consequently, for $r>1/2$ we set 
\begin{equation*}
\mathcal{M}^r_{\varepsilon}:= \left\{ \begin{array}{ll} L_{\mu_{\varepsilon}}^{2}\left( \mathbb{R}_{+};\mathbb{V}^{r}\right) & \text{for}\ \ep\in(0,1], \\ \{0\} & \text{when}\ \ep=0, \end{array} \right.  
\end{equation*}
and when $r=0$ set
\begin{equation*}
\mathcal{M}^0_{\varepsilon}:= \left\{ \begin{array}{ll} L_{\mu_{\varepsilon}}^{2}\left( \mathbb{R}_{+};\mathbb{X}^{2}\right) & \text{for}\ \ep\in(0,1], \\ \{0\} & \text{when}\ \ep=0. \end{array} \right.  
\end{equation*}
One can see from \cite[Lemma 5.1]{GMPZ10} that for $\ep_1\ge\ep_2>0$ and for fixed $r=0$ or $r>1/2$, there holds the continuous embedding $\mathcal{M}^{r}_{\ep_1}\hookrightarrow\mathcal{M}^{r}_{\ep_2}$.
As a matter of convenience, the inner-product in $\mathcal{M}^1_{\varepsilon}$ is given by 
\begin{align}
& \left\langle \binom{\eta_1}{\xi_1},\binom{\eta_2}{\xi_2} \right\rangle_{\mathcal{M}^1_{\varepsilon}}  \label{sc} \\
&=\int_0^\infty \mu_{\varepsilon}(s)\left( \left\langle \nabla\eta_1(s),\nabla\eta_2(s) \right\rangle_{L^2(\Omega)}+\alpha\left\langle \eta_1(s),\eta(s) \right\rangle_{L^2(\Omega)} \right) \diff s \notag \\
& + \int_0^\infty \mu_{\varepsilon}(s)\left( \left\langle \nabla_\Gamma\xi_1(s),\nabla_\Gamma\xi_2(s) \right\rangle_{L^2(\Gamma)} + \beta\left\langle \xi_1(s),\xi_2(s) \right\rangle_{L^2(\Gamma)} \right) \diff s. \notag
\end{align}
When it is convenient, we will use the notation 
\begin{align}
\mathcal{H}^0_{\varepsilon} & := \mathbb{X}^2\times\mathcal{M}^1_{\varepsilon} \notag \\ 
\mathcal{H}^1_{\varepsilon} & := \mathbb{V}^1\times\mathcal{M}^2_{\varepsilon}. \notag
\end{align}
Each space is equipped with the corresponding ``graph norm,'' whose square is defined by, for all $\ep\in[0,1]$ and $(U,\Phi)\in\mathcal{H}^i_\ep$, $i=0,1,$ 
\begin{equation*}
\left\|(U,\Phi)\right\|^2_{\mathcal{H}^0_{\varepsilon}}:=\left\|U\right\|^2_{\mathbb{X}^2}+\left\|\Phi\right\|^2_{\mathcal{M}^1_{\varepsilon}} \quad \text{and} \quad \left\|(U,\Phi)\right\|^2_{\mathcal{H}^1_{\varepsilon}}:=\left\|U\right\|^2_{\mathbb{V}^1}+\left\|\Phi\right\|^2_{\mathcal{M}^2_{\varepsilon}}.
\end{equation*}

For the kernel $\mu$, we take the following
assumptions (cf. e.g. \cite{CPS06,GPM98,GPM00}). Assume 
\begin{align}
& \mu\in C^1(\mathbb{R}_+)\cap L^1(\mathbb{R}_+),  \label{mu-1} \\
& \mu(s) \geq 0 \quad\text{for all}\quad s\geq 0,  \label{mu-2} \\
& \mu'(s) \leq 0 \quad\text{for all}\quad s\geq 0,  \label{mu-3} \\
& \mu'(s) + \delta\mu(s) \leq 0 \quad\text{for all} \quad s\geq 0 \quad\text{and some}\quad \delta>0.  \label{mu-4}
\end{align}
The assumptions (\ref{mu-1})-(\ref{mu-3}) are equivalent to assuming $k(s)$ be a bounded, positive, nonincreasing, convex function of class $\mathcal{C}^2$.
Moreover, assumption (\ref{mu-4}) guarantees exponential decay of the
function $\mu(s)$ while allowing a singularity at $s=0$. Assumptions (\ref{mu-1})-(\ref{mu-3}) are used in the literature (see \cite{CDGP-2010,CPS06,GPM98,Grasselli&Pata02-2} for example) to establish the
existence and uniqueness of continuous global weak solutions to a system of equations similar to (\ref{problemp-1}), (\ref{problemp-3}), but with Dirichlet boundary conditions. In the literature, assumption (\ref{mu-4}) is used to obtain a bounded absorbing set for the associated semigroup of solution operators.

For each $\varepsilon\in(0,1],$ define 
\begin{equation}\label{memory-4}
D(\mathrm{T}_{\varepsilon})=\left\{ \Phi\in\mathcal{M}^1_{\varepsilon}:\partial_s\Phi\in\mathcal{M}^1_{\varepsilon}, \Phi(0)=0 \right\}
\end{equation}
where (with an abuse of notation) $\partial_s\Phi$ is the distributional derivative of $\Phi$ and the equality $\Phi(0)=0$ is meant in the following sense
\begin{equation*}
\lim_{s\rightarrow0}\|\Phi(s)\|_{\mathbb{X}^{2}} = 0.
\end{equation*}
Then define the linear (unbounded) operator $\mathrm{T}_{\varepsilon}: D(\mathrm{T}_{\varepsilon})\rightarrow \mathcal{M}^1_{\varepsilon}$ by, for all $\Phi\in D(\mathrm{T}_{\varepsilon})$, 
\begin{equation*}
\mathrm{T}_{\varepsilon}\Phi=-\frac{\diff}{\diff s}\Phi.
\end{equation*}
For each $t\in[0,T]$, the equation 
\begin{equation}\label{memory-2}
\partial_t\Phi^t = \mathrm{T}_{\varepsilon}\Phi^t + U(t) 
\end{equation}
holds as an ODE in $\mathcal{M}^1_\ep$ subject to the initial condition
\begin{equation}\label{memory-3}
\Phi^0=\Phi_0\in\mathcal{M}^1_{\varepsilon}.
\end{equation}
Concerning the solution to the IVP (\ref{memory-2})-(\ref{memory-3}), we have the following proposition. 
The result is a generalization of \cite[Theorem 3.1]{Grasselli&Pata02-2}.

\begin{proposition}\label{t:generator-T}
For each $\varepsilon\in(0,1],$ the operator $\mathrm{T}_{\varepsilon}$ with domain $D(\mathrm{T}_{\varepsilon})$ is an infinitesimal generator of a strongly continuous semigroup of contractions on $\mathcal{M}^1_{\varepsilon}$, denoted $e^{\mathrm{T}_{\varepsilon} t}$.
\end{proposition}

We now have (cf. e.g. \cite[Corollary IV.2.2]{Pazy83}).

\begin{corollary}\label{t:memory-regularity-1}
When $U\in L^1([0,T];\mathbb{V}^1)$ for each $T>0$, then, for every $\Phi_0\in\mathcal{M}^1_{\varepsilon}$, the Cauchy problem 
\begin{equation}\label{memory-1}
\left\{ 
\begin{array}{ll}
\partial_t\Phi^t=\mathrm{T}_{\varepsilon}\Phi^t+U(t), & \text{for}\quad t>0, \\ 
\Phi^0=\Phi_0, & 
\end{array}
\right. 
\end{equation}
has a unique solution $\Phi\in C([0,T];\mathcal{M}^1_{\varepsilon})$ which can be explicitly given as (cf. \cite[Section 3.2]{CPS06} and \cite[Section 3]{Grasselli&Pata02-2}) 
\begin{equation}  \label{representation-formula-1}
\Phi^t(s)=\left\{ 
\begin{array}{ll}
\displaystyle\int_0^s U(t-y) \diff y, & \text{for }\quad 0<s\leq t, \\ 
\displaystyle\Phi_0(s-t) + \int_0^t U(t-y) \diff y, & \text{when}\quad s>t.
\end{array}
\right.
\end{equation}
(The interested reader can also see \cite[Section 3]{CPS06}, \cite[pp. 346--347]{GPM98} and \cite[Section 3]{Grasselli&Pata02-2} for more details concerning the case with static boundary conditions.)
\end{corollary}

Furthermore, we also know that, for each $\varepsilon\in(0,1]$, $\mathrm{T_\varepsilon}$ is the infinitesimal generator of a strongly continuous (the right-translation) semigroup of contractions on $\mathcal{M}^1_\varepsilon$ satisfying (\ref{operator-T-1}) below; in particular, $\mathrm{Range}(\mathrm{I}-\mathrm{T}_\varepsilon)=\mathcal{M}^1_\varepsilon$.

Following (\ref{mu-4}), there is the useful inequality.
(Also see \cite[see equation (3.4)]{CPS06} and \cite[Section 3, proof of Theorem]{Grasselli&Pata02-2}.)

\begin{corollary} \label{t:operator-T-1} 
There holds, for all $\Phi\in D(\mathrm{T}_\varepsilon)$, 
\begin{equation} \label{operator-T-1}
\left\langle \mathrm{T}_\varepsilon\Phi,\Phi \right\rangle_{\mathcal{M}^1_\varepsilon} \leq -\frac{\delta}{2\varepsilon}\|\Phi\|^2_{\mathcal{M}^1_\varepsilon}.
\end{equation}
\end{corollary}

Even though the embedding $\mathbb{V}^1\hookrightarrow\mathbb{X}^2$ is compact, it does not follow that the embedding $\mathcal{M}^{1}_{\ep}\hookrightarrow \mathcal{M}^{0}_{\ep}$ is also compact. 
Indeed, see \cite{Pata-Zucchi-2001} for a counterexample.
Moreover, this means the embedding $\mathcal{H}^1_\ep\hookrightarrow\mathcal{H}^0_\ep$ is not compact.
Such compactness between the ``natural phase spaces'' is essential to the construction of finite dimensional exponential attractors. 
To alleviate this issue we follow \cite{CPS06,GMPZ10} and define for any $\ep\in(0,1]$ the so-called {\em{tail function}} of $\Phi\in\mathcal{M}^{0}_\ep$ by, for all $\tau\ge0,$ 
\begin{equation*}
\mathbb{T}_{\ep}(\tau;\Phi) := \int\limits_{(0,1/\tau)\cup(\tau,\infty)} \ep\mu_\ep(s) \|\Phi(s)\|^2_{\mathbb{V}^1} \diff s,
\end{equation*}
With this we set, for $\ep\in(0,1],$ 
\begin{equation*}
\mathcal{K}^2_\ep := \{ \Phi\in\mathcal{M}^2_\ep : \partial_s\Phi\in\mathcal{M}^{0}_\ep,\ \Phi(0)=0,\ \sup_{\tau\ge1} \tau\mathbb{T}_\ep(\tau;\Phi)<\infty \}.
\end{equation*}
The space $\mathcal{K}^2_\ep$ is Banach with the norm whose square is defined by
\begin{equation}
\|\Phi\|^2_{\mathcal{K}^2_\ep} := \|\Phi\|^2_{\mathcal{M}^2_\ep} + \ep\|\partial_s\Phi\|^2_{\mathcal{M}^{0}_\ep} + \sup_{\tau\ge1} \tau\mathbb{T}_\ep(\tau;\Phi).  \label{new-norm}
\end{equation}
When $\ep=0,$ we set $\mathcal{K}^2_0=\{0\}.$
Importantly, for each $\ep\in(0,1]$, the embedding $\mathcal{K}^2_{\ep}\hookrightarrow\mathcal{M}^1_{\ep}$ is compact. (cf. \cite[Proposition 5.4]{GMPZ10}).
Hence, let us now also define the space 
\begin{align}
\mathcal{V}^1_{\varepsilon} & := \mathbb{V}^1\times\mathcal{K}^2_{\varepsilon}, \notag
\end{align}
and the desired compact embedding $\mathcal{V}^1_\ep\hookrightarrow\mathcal{H}^0_\ep$ holds.
Again, each space is equipped with the corresponding graph norm whose square is defined by, for all $\ep\in[0,1]$ and $(U,\Phi)\in\mathcal{V}^1_\ep$,
\begin{equation*}
\left\|(U,\Phi)\right\|^2_{\mathcal{V}^1_{\varepsilon}}:=\left\|U\right\|^2_{\mathbb{V}^1}+\left\|\Phi\right\|^2_{\mathcal{K}^2_{\varepsilon}}.
\end{equation*}

In regards to the system in Corollary \ref{t:memory-regularity-1} above, we will also call upon the following simple generalizations of \cite[Lemmas 3.3, 3.4, and 3.6]{CPS06}. 

\begin{lemma}  \label{what-1}
Let $\ep\in(0,1]$ and $\Phi_0\in D({\rm{T}}_\ep)$.
Assume there is $\rho>0$ such that, for all $t\ge0$, $\|U(t)\|_{\mathbb{V}^1}\le\rho$.
Then for all $t\ge0$,
\begin{align}
\ep\|{\rm{T}}_\ep\Phi^t\|^2_{\mathcal{M}^1_\ep} \le \ep e^{-\delta t}\|{\rm{T}}_\ep\Phi_0\|^2_{\mathcal{M}^1_\ep} + \rho^2\|\mu\|_{L^1(\mathbb{R}_+)}.  \notag
\end{align}
\end{lemma}

\begin{remark}  \label{r:what-trans}
The above result will also be needed later in the weaker space $\mathcal{M}^0_\ep$ (see Step 3 in the proof of Lemma \ref{t:to-C2}).
The result for the weaker space can be obtained by suitably transforming \eqref{memory-1}-\eqref{representation-formula-1} and applying an appropriate bound on $U$.
\end{remark}

\begin{lemma}  \label{what-2}
Let $\ep\in(0,1]$ and $\Phi_0\in D({\rm{T}}_\ep)$.
Assume there is $\rho>0$ such that, for all $t\ge0$, $\|U(t)\|_{\mathbb{V}^1}\le\rho$.
Then there is a constant $C>0$ such that, for all $t\ge0$,
\begin{align}
\sup_{\tau\ge1} \tau\mathbb{T}_\ep(\tau;\Phi^t) \le 2 \left( t+2 \right)e^{-\delta t} \sup_{\tau\ge1} \tau\mathbb{T}_\ep(\tau;\Phi_0) + C\rho^2.  \notag
\end{align}
\end{lemma}

Finally, we give a version of Lemma \ref{what-2} for compact intervals.

\begin{lemma}  \label{what-3}
Let $\ep\in(0,1]$, $T>0$, and $\Phi_0\in D({\rm{T}}_\ep)$.
Assume there is $\rho>0$ such that
\[
\int_0^T \|U(\tau)\|^2_{\mathbb{V}^1} \diff\tau \le \rho.
\]
Then there is a positive constant $C(T)$ such that, for all $t\in[0,T]$,
\begin{align}
\sup_{\tau\ge1} \tau\mathbb{T}_\ep(\tau;\Phi^t) \le C(T) \left( \rho + \sup_{\tau\ge1} \tau\mathbb{T}_\ep(\tau;\Phi_0) \right).  \notag
\end{align}
\end{lemma}

We now discuss the linear operator associated with the model problem.
In our case it is given by the following (note that in \cite[Section 3.1]{CPS06} the basic tool is the Laplacian with Dirichlet boundary conditions; in our case, the analogue operator turns out to be the so-called ``Wentzell'' Laplace operator).

\begin{proposition}
Let $\Omega $ be a bounded open set of $\mathbb{R}^{n}$ with Lipschitz boundary $\Gamma $. For $\alpha,\beta\ge0$, define the operator ${\mathrm{A_W^{\alpha,\beta}}}$ on $\mathbb{X}^{2},$
by 
\begin{align}
\mathrm{A_W^{\alpha,\beta}} & :=
\begin{pmatrix}
-\Delta+\alpha{\mathrm{I}} & 0 \\ 
\partial_{\mathbf{n}}(\cdot) & -\Delta_\Gamma + \beta{\mathrm{I}}
\end{pmatrix},  \label{A_Wentzell1}
\end{align}
with
\begin{equation}
D\left( {\mathrm{A_W^{\alpha,\beta}}}\right) :=\left\{ U=(u_1,u_2)^{\mathrm{tr}}\in \mathbb{V}^{1}:-\Delta u_{1}+\alpha u_1\in L^{2}\left( \Omega \right) ,\text{ }\partial_{\mathbf{n}} u_{1}-\Delta _{\Gamma }u_{2}+\beta u_2\in L^{2}\left( \Gamma \right) \right\} .  \label{A_Wentzell2}
\end{equation}
Then, $({\mathrm{A_W^{\alpha,\beta}}},D({\mathrm{A_W^{\alpha,\beta}}})) $ is self-adjoint and nonnegative operator on $\mathbb{X}^{2}$ whenever $\alpha,\beta\ge0$, and ${\mathrm{A_W^{\alpha,\beta}}}>0$ (is strictly positive) if either $\alpha>0$ or $\beta>0.$ 
Moreover, the resolvent operator $({\mathrm{I}}+{\mathrm{A_W^{\alpha,\beta}}}) ^{-1}\in \mathcal{L}(\mathbb{X}^{2}) $ is compact. 
If the boundary $\Gamma $ is of class $\mathcal{C}^{2},$ then $D({\mathrm{A_W^{\alpha,\beta}}}) =\mathbb{V}^{2}$ (see, e.g., \cite[Theorem 2.3]{CGGM10}). 
Indeed, for any $\alpha,\beta\ge 0$, the map $\Psi:U\mapsto {\mathrm{A_W^{\alpha,\beta}}}U,$ when viewed as a map from $\mathbb{V}^2$ into $\mathbb{X}^2=L^2(\Omega)\times L^2(\Gamma),$ is an isomorphism, and there exists a positive constant $C_*,$ independent of $U=(u,\psi)^{\mathrm{tr}}$, such that 
\begin{align}
C_*^{-1}\|U\|_{\mathbb{V}^2}\le\|\Psi(U)\|_{\mathbb{X}^2}\le C_*\|U\|_{\mathbb{V}^2},  \label{equiv}
\end{align}
for all $U\in\mathbb{V}^2$ (cf. Lemma \ref{t:appendix-lemma-3}).
\end{proposition}

We can refer the reader to \cite{CFGGGOR09} for an extensive survey of
recent results concerning the ``Wentzell'' Laplacian ${\mathrm{A_W^{\alpha,\beta}}}$. 

For the nonlinear terms, assume $f,g \in C^{1}(\mathbb{R})$ satisfy the growth assumptions: there exist positive constants $\ell_1$ and $\ell_2$, and $r_1,r_2\in[1,\frac{5}{2})$ such that for all $s\in \mathbb{R}$, 
\begin{align}
|f'(s)| \leq \ell_1(1+|s|^{r_1}),  \label{assm-1} \\
|g'(s)| \leq \ell_2(1+|s|^{r_2}).  \label{assm-2}
\end{align}
We also assume there are positive constants $M_f$ and $M_g$ so that for all $s\in\mathbb{R}$, 
\begin{align}
f'(s)>-M_f,  \label{assm-3} \\
g'(s)>-M_g.  \label{assm-4}
\end{align}
Consequently, \eqref{assm-3}-\eqref{assm-4} imply there are $\kappa_i>0$, $i=1,2,3,4,$ so that for all $s\in\mathbb{R},$
\begin{align}
f(s)s \ge -\kappa_1 s^2 - \kappa_2,  \label{cons-1} \\
g(s)s \ge -\kappa_3 s^2 - \kappa_4.  \label{cons-2}
\end{align}

\begin{remark}
Observe that here we do not allow for the critical polynomial growth exponent (of $5$) which appears in several works with static boundary conditions (cf. e.g. \cite{CDGP-2010,CPS06}). 
Indeed, in order for us to obtain a notion of strong solution (see Definition \ref{d:strong-solution} below), the arguments in the proof of Theorem \ref{t:strong-solutions} do not allow for $r_i\ge\frac{5}{2}, i=1,2.$
\end{remark}

We can follow \cite[Section 4]{CPS06} or, more precisely \cite{GPM98,GPM00} to deduce the existence and uniqueness of weak solutions in the
above class exploiting both semigroup methods and energy methods in the
framework of a Galerkin scheme which can be constructed for problems with
dynamic boundary conditions (see, \cite[Theorem 2.3]{CGGM10}). 

Constants appearing below are independent of $\varepsilon$ and $\omega$, unless specified otherwise, but may depend on various structural parameters such as $\alpha$, $\beta,$ $|\Omega|$, $|\Gamma|$, $\ell_f$ and $\ell_g$, and the constants may even change from line to line.
We denote by $Q(\cdot)$ a generic monotonically increasing function. 
We will use $\|B\|_{W}:=\sup_{\Upsilon\in B}\|\Upsilon\|_W$ to denote the ``size'' of the subset $B$ in the Banach space $W$.

%===========================================================
\section{Review of well-posedness and regularity}
%===========================================================

Here we provide some definitions and cite the relevant global well-posedness results concerning Problem {\textbf{P}}$_{\varepsilon}$.
For the remainder of this article we choose to set $n=3$, which is of course the most relevant physical dimension.

Below we will set $F:\mathbb{R}^2\rightarrow\mathbb{R}^2,$
\begin{equation}
F(U):=\begin{pmatrix}f(u) \\ \widetilde{g}(u)\end{pmatrix},  \label{func}
\end{equation}
where $\widetilde{g}(s):=g(s)-\omega\beta s$, for $s\in\mathbb{R}.$
(To offset $\widetilde{g}$, the term $\omega\beta u$ will be incorporated in the operator $\mathrm{A_{W}^{0,0}}$ as $\mathrm{A_{W}^{0,\beta}}.$)

\begin{definition}  \label{d:weak-solution} 
Let $\varepsilon\in(0,1]$, $\omega\in(0,1)$ and $T>0$. 
Given $U_0=(u_0,v_0)^{\mathrm{tr}}\in\mathbb{X}^2$ and $\Phi_0=(\eta_0,\xi_0)^{\mathrm{tr}}\in\mathcal{M}^1_{\varepsilon}$, the pair $U(t)=(u(t),v(t))^{{\mathrm{tr}}}$ and $\Phi^t=(\eta^t,\xi^t)^{{\mathrm{tr}}}$ satisfying 
\begin{align}
U &\in L^{\infty }([0,T];\mathbb{X}^{2})\cap L^{2}([0,T];\mathbb{V}^{1}), \label{defn-1} \\
u &\in L^{r_1}(\Omega\times[0,T]), \label{defn-2} \\ 
v &\in L^{r_2}(\Gamma\times[0,T]), \label{defn-3} \\ 
\Phi &\in L^{\infty }\left([0,T];\mathcal{M}^1_{\varepsilon}\right), \label{defn-4} \\ 
\partial_t U &\in L^2\left([0,T];(\mathbb{V}^1)^*\right) \oplus \left( L^{r_1'}(\Omega\times[0,T]) \times L^{r_2'}(\Gamma\times[0,T]) \right), \label{defn-5} \\ 
\partial_t \Phi &\in L^2\left( [0,T];H^{-1}_{\mu_{\varepsilon}}(\mathbb{R}_+;\mathbb{V}^1) \right), \label{defn-6}
\end{align}
is said to be a weak solution to Problem {\textbf{P}}$_{\varepsilon}$ if, $v(t)=u_{\mid\Gamma}(t)$ and $\xi^{t}=\eta^{t}_{\mid\Gamma}$ for almost all $t\in[0,T]$, and
for all $\Xi =(\varsigma ,\varsigma _{\mid \Gamma })^{\mathrm{tr}}\in \mathbb{V}^{1} \cap \left( L^{r_1}(\Omega) \times L^{r_2}(\Gamma) \right)$, $\Pi =(\rho ,\rho _{\mid \Gamma })^{\mathrm{tr}}\in \mathcal{M}^1_{\varepsilon}$, and for almost all $t\in[0,T]$, there holds, 
\begin{align}
& \left\langle \partial _{t}U(t),\Xi \right\rangle _{\mathbb{X}^{2}} + \omega
\langle \mathrm{{A_W^{0,\beta}}}U(t),\Xi \rangle _{\mathbb{X}^{2}} + \left\langle \Phi^t, \Xi \right\rangle_{\mathcal{M}^1_{\varepsilon}} + \left\langle
F\left( U(t)\right) ,\Xi \right\rangle _{\mathbb{X}^{2}}=0,  \label{weak-solution-1}
\end{align}
\begin{equation}\label{weak-solution-2}
\left\langle \partial _{t}\Phi ^{t},\Pi \right\rangle _{\mathcal{M}^1_{\varepsilon}}=\left\langle \mathrm{T}_{\varepsilon}\Phi ^{t},\Pi
\right\rangle _{\mathcal{M}^1_{\varepsilon}}+\left\langle U(t),\Pi
\right\rangle _{\mathcal{M}^1_{\varepsilon}},  
\end{equation}
in addition, 
\begin{equation}
U(0)=U_0 \quad\text{and}\quad \Phi^{0}=\Phi_0. \label{weak-solution-3}
\end{equation}
The function $[0,T]\ni t\mapsto (U(t),\Phi^t )$ is called a global weak solution if it is a weak solution for every $T>0$.
\end{definition}

\begin{remark}\label{r:trace-map}
When we have a weak solution to Problem {\textbf{P}}$_{\varepsilon}$, the above restrictions $u_{\mid\Gamma}(t)$ and $\eta_{\mid\Gamma}^t$ are well-defined by virtue of the Dirichlet trace map, ${\mathrm{tr_D}}:H^1(\Omega)\rightarrow H^{1/2}(\Gamma)$. However, this is not necessarily the case for $\partial_t U$.
\end{remark}

\begin{remark}  \label{r:continuity}
The continuity properties $U\in C([0,T];\mathbb{X}^2)$ follow from the classical embedding (cf. e.g. \cite[Lemma 5.51]{Tanabe79}),
\[
\{ \chi\in L^2([0,T];V), \ \partial_t\chi\in L^2([0,T];V') \} \hookrightarrow C([0,T];H),
\] 
where $H$ and $V$ are reflexive Banach spaces with continuous embeddings $V\hookrightarrow H \hookrightarrow V'$, the injection $V\hookrightarrow H$ being compact.
\end{remark}

\begin{definition} \label{d:strong-solution} 
The pair $U(t)=(u(t),v(t))^{\mathrm{tr}}$ and $\Phi^t=(\eta^t,\xi^t)^{\mathrm{tr}}$ is called a (global) strong solution of Problem {\textbf{P}}$_{\varepsilon}$ if it is a weak solution in the sense of Definition \ref{d:weak-solution}, and if it satisfies the following regularity properties:
\begin{align}
U &\in L^{\infty }([0,\infty);\mathbb{V}^{1})\cap L^2([0,\infty);\mathbb{V}^2), \label{strong-defn-1} \\
\Phi &\in L^{\infty }([0,\infty);\mathcal{M}^2_{\varepsilon}), \label{strong-defn-2} \\ 
\partial_t U &\in L^\infty\left([0,\infty);\mathbb{X}^2 \right)\cap L^2([0,\infty);\mathbb{V}^1), \label{strong-defn-3} \\ 
\partial_t \Phi &\in L^\infty\left( [0,\infty);\mathcal{M}^1_{\varepsilon} \right).  \label{strong-defn-4}
\end{align}
Therefore, $(U(t),\Phi^t)$ satisfies the equations (\ref{weak-solution-1})-(\ref{weak-solution-2}) almost everywhere, i.e., is a strong solution.
\end{definition}

\begin{theorem}[Weak solutions]\label{t:weak-solutions} 
Assume (\ref{mu-1})-(\ref{mu-3}) and (\ref{assm-1})-(\ref{assm-4}) hold. For each $\varepsilon\in(0,1]$, $\omega\in(0,1)$ and $T>0$, and for any $U_0=(u_0,v_0)^{ {\mathrm{tr}} }\in 
\mathbb{X}^2$ and $\Phi_0=(\eta_0,\xi_0)^{ {\mathrm{tr}} } \in\mathcal{M}^1_{\varepsilon}$, there exists a unique (global) weak solution to Problem {\textbf{P}}$_{\varepsilon}$ in the sense of Definition \ref{d:weak-solution} which depends continuously on the initial data in the following way; there exists a constant $C>0$, independent of $U_i$, $\Phi_i$, $i=1,2$, and $T>0$ in which, for all $t\in[0,T]$, there holds
\begin{align}
\left\| U_1(t)-U_2(t) \right\|_{\mathbb{X}^2} + \left\| \Phi^t_1-\Phi^t_2 \right\|_{\mathcal{M}^1_{\varepsilon}} \le \left( \left\| U_1(0)-U_2(0) \right\|_{\mathbb{X}^2} + \left\| \Phi^0_1-\Phi^0_2 \right\|_{\mathcal{M}^1_{\varepsilon}} \right) e^{Ct}. \label{cde}
\end{align}
\end{theorem}

\begin{proof}
Cf. \cite[Theorem 3.8]{Gal-Shomberg15-2} for existence and  \cite[Proposition 3.10]{Gal-Shomberg15-2} for (\ref{cde}).
\end{proof}

We conclude the preliminary results for Problem {\textbf{P}}$_\ep$ with the following. 

\begin{theorem}[Strong solutions]  \label{t:strong-solutions} 
Assume (\ref{mu-1})-(\ref{mu-3}) and (\ref{assm-1})-(\ref{assm-4}) hold. 
For each $\varepsilon\in(0,1]$, $\omega\in(0,1)$, and $T>0$, and for any $U_0=(u_0,v_0)^{ {\mathrm{tr}} }\in 
\mathbb{V}^1$ and $\Phi_0=(\eta_0,\xi_0)^{ {\mathrm{tr}} } \in\mathcal{M}^2_{\varepsilon}$, there exists a unique (global) strong solution to Problem {\textbf{P}}$_{\varepsilon}$ in the sense of Definition \ref{d:strong-solution}.
\end{theorem}

\begin{proof}
Cf. \cite[Theorem 3.11]{Gal-Shomberg15-2}.
\end{proof}

Here we recall some important aspects and relevant results for Problem {\textbf{P}}$_0$.
The interested reader can also see \cite{Gal12-1,Gal12-2,Gal-15Z,Gal&Warma10} for further details.

\begin{definition}  \label{d:weak-solution-0} 
Let $\omega\in(0,1)$ and $T>0$. 
Given $U_0=(u_0,v_0)^{\mathrm{tr}}\in\mathbb{X}^2$, the pair $U(t)=(u(t),v(t))^{{\mathrm{tr}}}$ satisfying 
\begin{align}
U &\in L^{\infty }([0,T];\mathbb{X}^{2})\cap L^{2}([0,T];\mathbb{V}^{1}), \label{defn-1-0} \\
u &\in L^{r_1}(\Omega\times[0,T]), \label{defn-2-0} \\ 
v &\in L^{r_2}(\Gamma\times[0,T]), \label{defn-3-0} \\ 
\partial_t U &\in L^2\left([0,T];(\mathbb{V}^1)^*\right) \oplus \left( L^{r_1'}(\Omega\times[0,T]) \times L^{r_2'}(\Gamma\times[0,T]) \right), \label{defn-5-0} 
\end{align}
is said to be a weak solution to Problem {\textbf{P}}$_{0}$ if, $v(t)=u_{\mid\Gamma}(t)$ for almost all $t\in[0,T]$, and
for all $\Xi =(\varsigma ,\varsigma _{\mid \Gamma })^{\mathrm{tr}}\in \mathbb{V}^{1} \cap \left( L^{r_1}(\Omega) \times L^{r_2}(\Gamma) \right)$,  and for almost all $t\in[0,T]$, there holds, 
\begin{align}
& \left\langle \partial _{t}U(t),\Xi \right\rangle _{\mathbb{X}^{2}} + \omega
\langle \mathrm{{A_W^{0,\beta}}}U(t),\Xi \rangle _{\mathbb{X}^{2}} + \left\langle
F\left( U(t)\right) ,\Xi \right\rangle _{\mathbb{X}^{2}}=0,  \label{weak-solution-1-0}
\end{align}
with, 
\begin{equation}
U(0)=U_0. \label{weak-solution-3-0}
\end{equation}
The function $[0,T]\ni t\mapsto U(t)$ is called a global weak solution if it is a weak solution for every $T>0$.
\end{definition}

We remind the reader of Remark \ref{r:trace-map} on the issue of traces. 

We conclude this section with the following. 

\begin{theorem}[Weak solutions]  \label{t:weak-solutions-0} 
Assume (\ref{assm-1})-(\ref{assm-4}) hold. 
For each $\omega\in(0,1)$ and $T>0$, and for any $U_0=(u_0,v_0)^{ {\mathrm{tr}} }\in \mathbb{X}^2$, there exists a unique (global) weak solution to Problem {\textbf{P}}$_0$ in the sense of Definition \ref{d:weak-solution-0} which depends continuously on the initial data as follows: there exists a constant $C>0$, independent of $U_1$ and $U_2$, and $T>0$ in which, for all $t\in[0,T]$, there holds
\begin{align}
\left\| U_1(t)-U_2(t) \right\|_{\mathbb{X}^2} \le \left\| U_1(0)-U_2(0) \right\|_{\mathbb{X}^2} e^{Ct}. \label{cde-0}
\end{align}
\end{theorem}

\begin{proof}
Cf. \cite[Theorem 2.2]{Gal12-1}.
\end{proof}

%===========================================================
\section{Asymptotic behavior and attractors}
%===========================================================

%===========================================================
\subsection{Preliminary estimates}
%===========================================================

Concerning Problem {\textbf{P}}$_{\varepsilon}$ and following directly from Theorem \ref{t:weak-solutions}, we have the first preliminary result for this section.

\begin{corollary}
Problem {\textbf{P}}$_{\varepsilon}$ defines a (nonlinear) strongly continuous semigroup $\mathcal{S}_{\varepsilon}(t)$ on the phase space $\mathcal{H}^0_{\varepsilon} = \mathbb{X}^2\times\mathcal{M}^1_{\varepsilon}$ by 
\begin{equation*}
\mathcal{S}_{\varepsilon}(t)\Upsilon_0 := \left( U(t),\Phi^t \right),
\end{equation*}
where $\Upsilon_0=(U_0,\Phi_0)\in \mathcal{H}^0_{\varepsilon}$ and $( U(t),\Phi^t )$ is the unique solution to Problem {\textbf{P}}$_{\varepsilon}$. 
The semigroup is Lipschitz continuous on $\mathcal{H}^0_{\varepsilon}$ via the continuous dependence estimate (\ref{cde}).
\end{corollary}

The next preliminary result concerns a uniform bound on the weak solutions. 
This result follows from an estimate which proves the existence of a bounded absorbing set for the semigroup of solution operators. 
This result provides a basic but important first step in showing the associated dynamical system is dissipative (cf. e.g. \cite{Babin&Vishik92,Temam88}).
It is important to note that throughout the remainder of this article, whereby we are now concerned with the asymptotic behavior of the solutions to Problem {\textbf{P}}$_{\varepsilon}$ and Problem {\textbf{P}}$_0$, 

\begin{description}
\item[A1] we will assume that (\ref{mu-4}) holds. 
\end{description} 

Additionally, we introduce a smallness criteria for certain parameters relating to the linear operator $\mathrm{A_{W}^{\alpha,\beta}}$ and the nonlinear map $F$.

\begin{description}
\item[A2] {\em{Smallness criteria:}} Fix $\varepsilon\in(0,1]$ and $\omega\in(0,1).$
Denote by $C_{\overline{\Omega}}$ the positive constant that arrises from the embedding $\mathbb{V}^1\hookrightarrow\mathbb{X}^2$; i.e., $\|U\|^2_{\mathbb{X}^2}\le C_{\overline{\Omega}}\|U\|^2_{\mathbb{V}^1}.$ The smallness criteria is that $\kappa_1,\kappa_3,\beta>0$ (cf. (\ref{A_Wentzell1}) and  (\ref{cons-1})-(\ref{cons-2})) satisfy
\begin{equation}  \label{smallness-criteria}
\max\{\kappa_1,\kappa_3+\beta\} < \omega C^{-1}_{\overline{\Omega}}.
\end{equation}
\end{description} 

As a final note, we remind the reader that all formal multiplication below can be rigorously justified using the Galerkin procedure developed in the proof of Theorem \ref{t:weak-solutions} of \cite{Gal-Shomberg15-2}.

\begin{lemma}  \label{weak-ball}
Let $\varepsilon\in(0,1]$ and $\omega\in(0,1)$.
In addition to the assumptions of Theorem \ref{t:weak-solutions}, assume (\ref{mu-4}) holds and that $\kappa_1,\kappa_3,\beta>0$ satisfy the smallness criteria (\ref{smallness-criteria}). 
For all $R>0$ and $\Upsilon_0=(U_0,\Phi_0)\in\mathcal{H}^0_{\varepsilon}=\mathbb{X}^2\times\mathcal{M}^1_{\varepsilon}$ with $\|\Upsilon_0\|_{\mathcal{H}^0_{\varepsilon}}\le R$ for all $\varepsilon\in(0,1]$, there exist positive constants $\nu_0=\nu_0(\omega,C_{\overline{\Omega}}, \kappa_1,\kappa_3,\beta,\delta)$ and $P_0=P_0(\kappa_2,\kappa_4,\nu_0)$, and there is a positive monotonically increasing function $Q(\cdot)$ each independent of $\varepsilon$, in which, for all $t\ge0$,
\begin{align}
\left\| \left( U(t),\Phi^t \right) \right\|^2_{\mathcal{H}^0_{\varepsilon}} \le Q(R)e^{-\nu_0t} + P_0.  \label{weak-decay} 
\end{align}
Moreover, the set 
\begin{equation}  \label{ball-0}
\mathcal{B}^0_{\varepsilon}:=\left\{ (U,\Phi)\in \mathcal{H}^0_{\varepsilon} : \left\| (U,\Phi) \right\|_{\mathcal{H}^0_{\varepsilon}} \le \sqrt{P_0+1} \right\}.
\end{equation}
is absorbing and positively invariant for the semigroup $\mathcal{S}_{\varepsilon}(t).$
\end{lemma}

\begin{proof}
Let $\varepsilon\in(0,1]$ and $\omega\in(0,1)$.
Let $\Upsilon_0=(U_0,\Phi_0)\in\mathcal{H}^0_{\varepsilon}=\mathbb{X}^2\times\mathcal{M}^1_{\varepsilon}.$
From the equations (\ref{weak-solution-1}) and (\ref{weak-solution-2}), we take the corresponding weak solution $\Xi=U(t)$ and $\Pi(s)=\Phi^t(s).$
We then obtain the identities 
\begin{align}
\left\langle \partial_t U,U \right\rangle_{\mathbb{X}^2} + \omega\left\langle \mathrm{A_{W}^{0,\beta}}U,U \right\rangle_{\mathbb{X}^2} + \left\langle \Phi^t,U \right\rangle_{\mathcal{M}^1_{\varepsilon}} + \left\langle F(U),U \right\rangle_{\mathbb{X}^2} = 0,  \label{fus-1} 
\end{align}
and
\begin{align}
\left\langle \partial_t\Phi^t,\Phi^t \right\rangle_{\mathcal{M}^1_{\varepsilon}} = \left\langle  \mathrm{T_{\varepsilon}}\Phi^t,\Phi^t \right\rangle_{\mathcal{M}^1_{\varepsilon}} + \left\langle U,\Phi^t \right\rangle_{\mathcal{M}^1_{\varepsilon}}.  \label{fus-2}
\end{align}
Observe,
\begin{align}
& \left\langle \partial_t U,U \right\rangle_{\mathbb{X}^2} = \frac{1}{2}\frac{\diff}{\diff t}\left\| U \right\|^2_{\mathbb{X}^2},  \label{fus-3} \\
& \left\langle \mathrm{A_{W}^{0,\beta}}U,U \right\rangle_{\mathbb{X}^2} = \|\nabla u\|^2_{L^2(\Omega)}+\|\nabla_\Gamma u\|^2_{L^2(\Gamma)}+\beta\|u\|^2_{L^2(\Gamma)}, \label{fus-4}
\end{align}
and
\begin{align}
\left\langle \partial_t\Phi^t,\Phi^t \right\rangle_{\mathcal{M}^1_{\varepsilon}} = \frac{1}{2}\frac{\diff}{\diff t}\left\| \Phi^t \right\|^2_{\mathcal{M}^1_{\varepsilon}}.  \label{fus-5}
\end{align}
Combining (\ref{fus-1})-(\ref{fus-5}) produces the differential identity, which holds for almost all $t\ge0$,
\begin{align}
& \frac{1}{2}\frac{\diff}{\diff t}\left\{ \left\| U \right\|^2_{\mathbb{X}^2} + \left\| \Phi^t \right\|^2_{\mathcal{M}^1_{\varepsilon}} \right\}   \label{fus-6} \\
& + \omega \left( \|\nabla u\|^2_{L^2(\Omega)} + \|\nabla_\Gamma u\|^2_{L^2(\Gamma)} + \beta\|u\|^2_{L^2(\Gamma)} \right)  \notag \\ 
& - \left\langle \mathrm{T}_{\varepsilon}\Phi^t,\Phi^t \right\rangle_{\mathcal{M}^1_{\varepsilon}} + \left\langle F(U),U \right\rangle_{\mathbb{X}^2} = 0.  \notag
\end{align}
Because of assumption (\ref{mu-4}),  we may directly apply (\ref{operator-T-1}) from Corollary \ref{t:operator-T-1}; i.e., 
\begin{equation}  \label{fus-7}
-\left\langle \mathrm{T}_{\varepsilon}\Phi^t,\Phi^t \right\rangle_{\mathcal{M}^1_{\varepsilon}} \ge \frac{\delta}{2\varepsilon}\left\|\Phi^t\right\|^2_{\mathcal{M}^1_{\varepsilon}}.
\end{equation}
With (\ref{cons-1}) and (\ref{cons-2}), we know 
\begin{align}
\left\langle F(U),U \right\rangle_{\mathbb{X}^2} & \ge -\kappa_1\|u\|^2_{L^2(\Omega)}-\left(\kappa_3+\omega\beta\right)\|u\|^2_{L^2(\Gamma)} - \left( \kappa_2+\kappa_4 \right)   \label{fus-8} \\
& \ge -\kappa_1\|u\|^2_{L^2(\Omega)}-\left(\kappa_3 + \beta\right)\|u\|^2_{L^2(\Gamma)} - \left( \kappa_2+\kappa_4 \right)  \notag \\
& = -C_F\left\|U\right\|^2_{\mathbb{X}^2} - \left( \kappa_2+\kappa_4 \right),  \notag
\end{align}
where $C_F:=\max\{\kappa_1,\kappa_3+\beta\}.$ 
Finally, due the embedding $\mathbb{V}^1\hookrightarrow\mathbb{X}^2,$ we have 
\begin{equation}\label{fus-9}
C_{\overline{\Omega}}^{-1}\|U\|^2_{\mathbb{X}^2}\le \|U\|^2_{\mathbb{V}^1},
\end{equation}
for some $C_{\overline{\Omega}}>0.$
Hence, (\ref{fus-6})-(\ref{fus-9}) yields the differential inequality (minimizing the left-hand side by setting $\varepsilon=1$),
\begin{align}
& \frac{\diff}{\diff t}\left\{ \left\| U \right\|^2_{\mathbb{X}^2} + \left\| \Phi^t \right\|^2_{\mathcal{M}^1_{\varepsilon}} \right\}   \notag \\
& + 2\left( \omega C^{-1}_{\overline{\Omega}} - C_F \right) \left\| U \right\|^2_{\mathbb{X}^2} + \delta \left\|\Phi^t\right\|^2_{\mathcal{M}^1_{\varepsilon}}  \notag \\ 
& \le 2\left(\kappa_2+\kappa_4\right).  \notag
\end{align}
Observe, by the smallness criteria (\ref{smallness-criteria}) there holds, 
\begin{equation*}
\omega C^{-1}_{\overline{\Omega}} - C_F>0.
\end{equation*}
Thus we arrive at the differential inequality, which holds for almost all $t\ge0,$
\begin{align}
& \frac{\diff}{\diff t}\left\{ \left\| U \right\|^2_{\mathbb{X}^2} + \left\| \Phi^t \right\|^2_{\mathcal{M}^1_{\varepsilon}} \right\}   \label{fus-10} + m_0 \left( \left\| U \right\|^2_{\mathbb{X}^2} + \left\|\Phi^t\right\|^2_{\mathcal{M}^1_{\varepsilon}} \right)  \\ 
& \le C.  \notag
\end{align}
where $m_0:=\min\{ 2( \omega C^{-1}_{\overline{\Omega}} - C_F),\delta \}>0$, and $C>0$ depends only on $\kappa_2$ and $\kappa_4$.
(The absolute continuity of the mapping $t\mapsto \left\| U(t) \right\|^2_{\mathbb{X}^2}+\left\| \Phi^t \right\|^2_{\mathcal{M}^1_{\varepsilon}}$ can be established as in \cite[Lemma III.1.1]{Temam88}, for example.)
After applying a suitable Gr\"{o}nwall inequality, the estimate (\ref{weak-decay}) follows with $\nu_0=m_0$ and $P_0=\frac{C}{m_0};$ indeed, (\ref{fus-10}) yields, for all $t\ge 0,$
\begin{align}
& \left\| U(t) \right\|^2_{\mathbb{X}^2} + \left\| \Phi^t \right\|^2_{\mathcal{M}^1_{\varepsilon}}  \label{fus-11} \\
& \le e^{-m_0 t} \left( \left\| U_0 \right\|^2_{\mathbb{X}^2} + \left\| \Phi_0 \right\|^2_{\mathcal{M}^1_{\varepsilon}} \right) + P_0.  \notag
\end{align}
Now we see (\ref{weak-decay}) holds for any $R>0$ and $\Upsilon_0=(U_0,\Phi_0)\in\mathcal{H}^0_{\varepsilon}$ such that $\|\Upsilon_0\|_{\mathcal{H}^0_{\varepsilon}}\le R$ for all $\varepsilon\in(0,1]$. 

The existence of the bounded set $\mathcal{B}^0_{\varepsilon}$ in $\mathcal{H}^0_{\varepsilon}$ that is absorbing and positively invariant for $\mathcal{S}_{\varepsilon}(t)$ follows from (\ref{fus-11}) (cf. e.g. \cite[Proposition 2.64]{Milani&Koksch05}).
Given any nonempty bounded subset $B$ in $\mathcal{H}^0_{\varepsilon}\setminus\mathcal{B}^0_{\varepsilon}$, then we have that $\mathcal{S}_{\varepsilon}(t)B\subseteq\mathcal{B}^0$, in $\mathcal{H}^0_{\varepsilon}$, for all $t\ge t_0$ where
\begin{equation}\label{time0}  
t_0\ge\frac{1}{m_0}\ln\left( \|B\|^2_{\mathcal{H}^0_{\varepsilon}} \right).
\end{equation}
(Observe, $t_0>0$ because $\|B\|_{\mathcal{H}^0_{\varepsilon}}>1$.) 
This finishes the proof.
\end{proof}

\begin{corollary}
From (\ref{weak-decay}) it follows that for each $\varepsilon\in(0,1]$ and $\omega\in(0,1)$, any weak solution $(U(t),\Phi^t)$ to Problem {\textbf{P}}$_{\varepsilon}$, according to Definition \ref{d:weak-solution}, is bounded uniformly in $t$. 
Indeed, for all $\Upsilon_0\in\mathcal{H}^0_\varepsilon$,
\begin{equation}\label{weak-bound}
\limsup_{t\rightarrow+\infty} \left\|\mathcal{S}_\varepsilon(t)\Upsilon_0\right\|_{\mathcal{H}^0_{\varepsilon}} \le \widetilde{P}_0,
\end{equation}
where $\widetilde{P}_0$ depends on $P_0$ and the initial datum.
\end{corollary}

\begin{corollary}
Problem {\textbf{P}}$_{\varepsilon}$ defines a (nonlinear) strongly continuous semigroup $\mathcal{S}_{\varepsilon}(t)$ on the phase space $\mathcal{H}^0_{\varepsilon} = \mathbb{X}^2\times\mathcal{M}^1_{\varepsilon}$ by 
\begin{equation*}
\mathcal{S}_{\varepsilon}(t)\Upsilon_0 := \left( U(t),\Phi^t \right),
\end{equation*}
where $\Upsilon_0=(U_0,\Phi_0)\in \mathcal{H}^0_{\varepsilon}$ and $( U(t),\Phi^t )$ is the unique solution to Problem {\textbf{P}}$_{\varepsilon}$. 
The semigroup is Lipschitz continuous on $\mathcal{H}^0_{\varepsilon}$ via the continuous dependence estimate (\ref{cde}).
\end{corollary}

\begin{remark}
Thanks to the uniformity of the above estimates with respect to the perturbation parameter $\ep$, it is easy to see that there exists a bounded absorbing set $\mathcal{B}^0_0$ for the semigroup $\mathcal{S}_0:\mathcal{H}^0_0=\mathbb{X}^2\rightarrow\mathbb{X}^2$ generated by the weak solutions of Problem {\textbf{P}}$_0$.
Moreover, we also easily see that Problem {\textbf{P}}$_0$ defines a semigroup $\mathcal{S}_0(t):\mathcal{H}^0_0=\mathbb{X}^2\rightarrow\mathbb{X}^2$ by $\mathcal{S}_0(t)U_0:=U(t).$
(See the references mentioned above for further details.)
\end{remark}

%===========================================================
\subsection{Exponential attractors}
%===========================================================

Exponential attractors (sometimes called inertial sets) are positively invariant sets possessing finite fractal dimension that attract bounded subsets of their basin of attraction exponentially fast. 
This section will focus on the existence of exponential attractors.
The existence of an exponential attractor depends on certain properties of the semigroup; namely, the smoothing property for the difference of any two trajectories and the existence of a more regular bounded absorbing set in the phase space (see e.g. \cite{EFNT95,EMZ00,GGMP05} and in particular \cite{CPS06,GMPZ10}).
The basin of attraction will be discussed in the next section.

The main result of this section is the following.

\begin{theorem}  \label{t:exponential-attractors} 
For each $\varepsilon \in [0,1]$ and $\omega\in(0,1)$, the dynamical system $\left( \mathcal{S}_{\varepsilon},\mathcal{H}^0_{\varepsilon }\right) $
associated with Problem P$_\varepsilon$ admits an exponential
attractor $\mathfrak{M}_{\varepsilon }$ compact in $\mathcal{H}^0_{\varepsilon},$ and bounded in $\mathcal{V}^1_{\varepsilon }.$ 
Moreover, there hold:

(i) For each $t\geq 0$, $\mathcal{S}_{\varepsilon }(t)\mathfrak{M}_{\varepsilon}\subseteq \mathfrak{M}_{\varepsilon }$.

(ii) The fractal dimension of $\mathfrak{M}_{\varepsilon }$ with respect to
the metric $\mathcal{H}_{\varepsilon }^0$ is finite, uniformly in $\varepsilon$; namely,
\begin{equation*}
\dim _{\rm{F}}\left( \mathfrak{M}_{\varepsilon },\mathcal{H}^0_{\varepsilon }\right)
\leq C<\infty ,
\end{equation*}
for some positive constant $C$ independent of $\varepsilon$.

(iii) There exist $\varrho >0$ and a positive nondecreasing function $Q$ such that, for all $t\geq 0$, 
\begin{equation*}
\mathrm{dist}_{\mathcal{H}^0_{\varepsilon }}(\mathcal{S}_{\varepsilon }(t)B,\mathfrak{M}_{\varepsilon })\leq Q(\Vert B\Vert _{\mathcal{H}^0_{\varepsilon }})e^{-\varrho t},
\end{equation*}%
for every nonempty bounded subset $B$ of $\mathcal{H}^0_{\varepsilon }.$
\end{theorem}

\begin{remark}
Above, the fractal dimension of $\mathfrak{M}_{\varepsilon}$ in $\mathcal{H}^0_{\varepsilon}$ is given by
\begin{equation*}
\dim _{\mathrm{F}}\left(\mathfrak{M}_{\varepsilon },\mathcal{H}^0_{\varepsilon}\right):=\limsup_{r\rightarrow 0}\frac{\ln \mu _{\mathcal{H}^0_{\varepsilon }}(\mathfrak{M}_{\varepsilon },r)}{-\ln r}<\infty
\end{equation*}
where $\mu _{\mathcal{H}^0_{\varepsilon }}(\mathcal{X},r)$ denotes the
minimum number of $r$-balls from $\mathcal{H}^0_{\varepsilon }$ required to
cover $\mathcal{X}$.
\end{remark}

The proof of Theorem \ref{t:exponential-attractors} follows from the application of an abstract result reported here for our problem (see e.g. \cite{CPS06,GMPZ10}; cf. also Remark \ref{rem_att} below).

\begin{proposition}
\label{abstract1}
Let $\left( \mathcal{S}_{\varepsilon },\mathcal{H}^0_{\varepsilon
}\right) $ be a dynamical system for each $\varepsilon\in[0,1]$. Assume the
following hypotheses hold:

\begin{enumerate}
\item[(C1)] There exists a bounded absorbing set $\mathcal{B}_{\varepsilon
}^{1}\subset \mathcal{V}^1_{\varepsilon }$ which is positively invariant for $\mathcal{S}_{\varepsilon }(t).$ 
More precisely, there exists a time $t_{1}>0,$ uniform in $\varepsilon$, such that
\begin{equation*}
\mathcal{S}_{\varepsilon }(t)\mathcal{B}_{\varepsilon }^{1}\subset \mathcal{B}_{\varepsilon }^{1}
\end{equation*}
for all $t\geq t_{1}$ where $\mathcal{B}_{\varepsilon }^{1}$ is endowed with
the topology of $\mathcal{H}^0_{\varepsilon }.$

\item[(C2)] There is $t^{\ast }\geq t_{1}$ such that the map $\mathcal{S}_{\varepsilon}(t^{\ast })$ admits the decomposition, for each $\varepsilon \in (0,1]$ and for all $\Upsilon _{0},\Xi_{0}\in \mathcal{B}_{\varepsilon }^{1}$, 
\begin{equation*}
\mathcal{S}_{\varepsilon }(t^{\ast })\Upsilon _{0} - \mathcal{S}_{\varepsilon }(t^{\ast })\Xi_{0} = L_{\varepsilon }(\Upsilon_{0},\Xi_{0}) + R_{\varepsilon }(\Upsilon_{0},\Xi_{0})
\end{equation*}
where, for some constants $\alpha ^{\ast }\in (0,\frac{1}{2})$ and $\Lambda
^{\ast }=\Lambda ^{\ast }(\Omega ,t^{\ast },\omega)\geq 0$, the following hold:
\begin{equation}
\left\| L_{\varepsilon }(\Upsilon_{0},\Xi_{0}) \right\|_{\mathcal{H}^0_{\varepsilon }} \leq \alpha ^{\ast } \left\| \Upsilon _{0}-\Xi_{0} \right\|_{\mathcal{H}^0_{\varepsilon}}  \label{difference-decomposition-L}
\end{equation}
and
\begin{equation}
\left\| R_{\varepsilon}(\Upsilon_{0},\Xi_{0}) \right\|_{\mathcal{V}^1_{\varepsilon }} \leq \Lambda ^{\ast } \left\| \Upsilon_{0}-\Xi_{0} \right\|_{\mathcal{H}^0_{\varepsilon }}.  \label{difference-decomposition-K}
\end{equation}

\item[(C3)] The map
\begin{equation*}
(t,\Upsilon) \mapsto \mathcal{S}_{\varepsilon }(t)\Upsilon:[t^{\ast },2t^{\ast }]\times \mathcal{B}_{\varepsilon }^{1}\rightarrow \mathcal{B}_{\varepsilon }^{1}
\end{equation*}
is Lipschitz continuous on $\mathcal{B}_{\varepsilon }^{1}$ in the topology
of $\mathcal{H}^0_{\varepsilon }$.
\end{enumerate}

Then, $(\mathcal{S}_{\varepsilon },\mathcal{H}^0_{\varepsilon }) $ possesses an exponential attractor $\mathfrak{M}_{\varepsilon }$ in $\mathcal{B}_{\varepsilon }^{1}.$
\end{proposition}

We now prove the hypotheses of Proposition \ref{abstract1} and we again remind the reader that for the remainder of the article, we assume that the smallness criteria (\ref{smallness-criteria}) holds, in addition to the assumption (\ref{mu-4}).
We begin with the perturbation Problem {\textbf{P}}$_\ep$.
The results for the singular Problem {\textbf{P}}$_0$ will follow.

\begin{lemma}  \label{t:to-C1} 
Condition (C1) holds for each $\varepsilon\in(0,1]$ and $\omega\in(0,1)$. Moreover, for all $R>0$ and $\Upsilon_0=(U_0,\Phi_0)\in\mathcal{V}^1_{\varepsilon}=\mathbb{V}^1\times\mathcal{K}^2_{\varepsilon}$ with $\|\Upsilon_0\|_{\mathcal{V}^1_\ep}\le R$ for all $\varepsilon\in(0,1]$, there exists a positive constant $P_1=P_1(\nu_1,\widetilde{P}_0)$ and a positive monotonically increasing function $Q(\cdot)$, each independent of $\ep$, such that, for all $t\ge0$,
\begin{align}
\left\| \left( U(t),\Phi^t \right) \right\|^2_{\mathcal{V}^1_{\varepsilon}} \le Q(R) e^{-\min\{\delta,1\}t} \left( t+1 \right) + 2P_1.  \label{strong-decay} 
\end{align}
\end{lemma}

\begin{proof}
Let $\varepsilon\in(0,1]$, $\omega\in(0,1)$ and $\Upsilon_0=(U_0,\Phi_0)\in\mathcal{V}^1_{\varepsilon}=\mathbb{V}^1\times\mathcal{K}^2_{\varepsilon}.$
For all $s,t\ge0$, let $Z(t)={\rm{A_W^{\alpha,\beta}}}U(t)$ and $\Theta^t(s)={\rm{A_W^{\alpha,\beta}}}\Phi^t(s)$.
In equations (\ref{weak-solution-1})-(\ref{weak-solution-2}), take $\Xi=Z(t)$ and $\Pi=\Theta^t(s)$.
Proceeding as in \cite[proof of Theorem 3.11]{Gal-Shomberg15-2} (however, this time we are able to enjoy the uniform bounds (\ref{strong-defn-1})), we obtain the identities,
\begin{equation}
\left\langle \partial _{t}U,{Z}\right\rangle _{\mathbb{X}^{2}}+\omega\left\langle \mathrm{A_{W}^{0,\beta}}U,{Z}\right\rangle _{\mathbb{X}^{2}}+\left\langle \Phi^{t}, Z\right\rangle _{\mathcal{M}^{1}_{\varepsilon}}+\left\langle F(U),{Z}\right\rangle _{\mathbb{X}^{2}}=0,  \label{qest3}
\end{equation}
and 
\begin{equation}
\left\langle \partial _{t}\Phi^{t},\Theta^{t}\right\rangle _{\mathcal{M}^{1}_{\varepsilon}}=\left\langle \mathrm{T}_{\varepsilon}\Phi^{t},\Theta^{t}\right\rangle _{\mathcal{M}^{1}_{\varepsilon}}+\left\langle U,\Theta^{t}\right\rangle _{\mathcal{M}^{1}_{\varepsilon}}.  \label{qest4}
\end{equation}
These two identities may be combined together after we observe that, from the definition of the product given in (\ref{sc}), 
\begin{align}
\left\langle \Phi^{t}, Z\right\rangle _{\mathcal{M}^{1}_{\varepsilon}}& =\int_{0}^{\infty }\mu_{\varepsilon}(s)\left\langle \Phi^{t}\left( s\right),Z\right\rangle _{\mathbb{V}^{1}}\diff s  \label{qest2} \\
& =\int_{0}^{\infty }\mu_{\varepsilon}(s)\left\langle 
\mathrm{A_{W}^{\alpha,\beta}}\Phi^{t}\left( s\right)
,Z\right\rangle _{\mathbb{X}^{2}}\diff s  \notag \\
& =\int_{0}^{\infty }\mu_{\varepsilon}(s)\left\langle \mathrm{A_{W}^{\alpha,\beta}}\Phi^{t}\left( s\right) ,\mathrm{A_{W}^{\alpha,\beta}}U\right\rangle _{\mathbb{X}^{2}}\diff s  \notag \\
& =\int_{0}^{\infty }\mu_{\varepsilon}(s)\left\langle \Theta^{t}\left( s\right),\mathrm{A_{W}^{\alpha,\beta}}U\right\rangle _{\mathbb{X}^{2}}\diff s  \notag \\
& =\int_{0}^{\infty }\mu_{\varepsilon}(s)\left\langle \Theta^{t}\left( s\right) ,U\right\rangle _{\mathbb{V}^{1}}\diff s  \notag \\
& =\left\langle U,\Theta^{t}\right\rangle _{\mathcal{M}^{1}_{\varepsilon}}.  \notag
\end{align}
Now inserting (\ref{qest2}) into (\ref{qest3}) and adding the result to (\ref{qest4}), we obtain the identity
\begin{align}
&\left\langle \partial _{t}U,{Z}\right\rangle _{\mathbb{X}^{2}} + \omega\left\langle \mathrm{A_{W}^{0,\beta}}U,{Z}\right\rangle _{\mathbb{X}^{2}} + \left\langle \partial _{t}\Phi^{t},\Theta^{t}\right\rangle _{\mathcal{M}^{1}_{\varepsilon}}\label{qest5} \\
& - \left\langle \mathrm{T}_{\varepsilon}\Phi^{t},\Theta^{t}\right\rangle _{\mathcal{M}^{1}_{\varepsilon}} + \left\langle F(U),{Z}\right\rangle _{\mathbb{X}^{2}} = 0.  \notag 
\end{align}
Next we write 
\begin{align}
\left\langle \partial _{t}U,{Z}\right\rangle _{\mathbb{X}^{2}} & =\left\langle \partial _{t}U,{\mathrm{A^{\alpha,\beta}_{W}}}{U}\right\rangle _{\mathbb{X}^{2}} \label{qest10} \\
& =\left\langle \partial _{t}U,{U}\right\rangle _{\mathbb{V}^{1}}  \notag \\
&=\frac{1}{2}\frac{\diff}{\diff t} \left\| U \right\|^2_{\mathbb{V}^1},  \notag
\end{align}
\begin{align} 
\omega\left\langle {\mathrm{A_{W}^{0,\beta}}}U,{Z}\right\rangle _{\mathbb{X}^{2}} & = \omega\left( \left\langle \nabla u,\nabla z \right\rangle_{L^2(\Omega)} + \left\langle \nabla_\Gamma u,\nabla_\Gamma z \right\rangle_{L^2(\Gamma)} + \beta\left\langle v,z \right\rangle_{L^2(\Gamma)} \right)  \label{qest11}  \\ 
& = \omega\left\langle \mathrm{A_{W}^{\alpha,\beta}}U,Z \right\rangle_{\mathbb{X}^2} - \omega\alpha\left\langle u,z\right\rangle_{L^2(\Omega)} \notag  \\ 
& = \omega\left\| Z \right\|^2_{\mathbb{X}^2} - \omega\alpha\left\langle u,z\right\rangle_{L^2(\Omega)}  \notag, 
\end{align}
and 
\begin{align}
\left\langle \partial _{t}\Phi^{t},\Theta^{t}\right\rangle _{\mathcal{M}^{1}_{\varepsilon}} & = \int_0^\infty \mu_{\varepsilon}(s) \left\langle \partial_t\Phi^t(s),\Theta^{t}(s) \right\rangle_{\mathbb{V}^1}\diff s  \label{qest12} \\
& = \int_0^\infty \mu_{\varepsilon}(s) \left\langle \partial_t{\mathrm{A^{\alpha,\beta}_{W}}}\Phi^t(s),\Theta^t(s) \right\rangle_{\mathbb{X}^2}\diff s  \notag \\
& = \int_0^\infty \mu_{\varepsilon}(s) \left\langle \partial_t\Theta^t(s),\Theta^t(s) \right\rangle_{\mathbb{X}^2}\diff s  \notag \\
& = \left\langle \partial_t\Theta^t,\Theta^t \right\rangle_{\mathcal{M}^0_\ep}  \notag \\
& = \frac{1}{2}\frac{\diff}{\diff t}\left\| \Theta^t \right\|^2_{\mathcal{M}^0_\ep}.  \notag
\end{align}
Combining (\ref{qest5})-(\ref{qest12}) brings us to the differential identity, which holds for almost all $t\ge0$, 
\begin{align}
&\frac{1}{2}\frac{\diff}{\diff t}\left\{ \left\| U \right\|^2_{\mathbb{V}^1} + \left\| \Theta^t \right\|^2_{\mathcal{M}^0_\ep} \right\}  \label{qest13} \\
& + \omega\left\| Z \right\|^2_{\mathbb{X}^2} - \left\langle \mathrm{T}_{\varepsilon}\Phi^{t},\Theta^{t}\right\rangle _{\mathcal{M}^{1}_{\varepsilon}} + \left\langle F(U),{Z}\right\rangle _{\mathbb{X}^{2}}  \notag \\ 
& = \omega\alpha\left\langle u,z \right\rangle_{L^2(\Omega)}.  \notag
\end{align}
With assumption (\ref{mu-4}) we are able to estimate the following 
\begin{align}
\left\langle {\mathrm{T}_{\varepsilon}}\Phi^t,\Theta^t \right\rangle_{\mathcal{M}^1_\ep} & = \int_0^\infty\mu_{\varepsilon}(s)\left\langle {\mathrm{T}_{\varepsilon}}\Phi^t(s),\Theta^t(s) \right\rangle_{\mathbb{V}^1}\diff s  \label{qest14} \\ 
& = \int_0^\infty\mu_{\varepsilon}(s)\left\langle {\mathrm{T}_{\varepsilon}}{\mathrm{A^{\alpha,\beta}_{W}}}\Phi^t(s),\Theta^t(s) \right\rangle_{\mathbb{X}^2}\diff s  \notag \\ 
& = \int_0^\infty\mu_{\varepsilon}(s)\left\langle {\mathrm{T}_{\varepsilon}}\Theta^t(s),\Theta^t(s) \right\rangle_{\mathbb{X}^2}\diff s  \notag \\ 
& = -\int_0^\infty\mu_{\varepsilon}(s)\left\langle \frac{\diff}{\diff s}\Theta^t(s),\Theta^t(s) \right\rangle_{\mathbb{X}^2}\diff s  \notag \\ 
& = -\frac{1}{2}\int_0^\infty\mu_{\varepsilon}(s)\frac{\diff}{\diff s}\left\|\Theta^t(s)\right\|^2_{\mathbb{X}^2}\diff s  \notag \\ 
& = \underbrace{-\frac{1}{2}\int_0^\infty\frac{\diff}{\diff s}\left(\mu_{\varepsilon}(s)\left\|\Theta^t(s)\right\|^2_{\mathbb{X}^2}\right)\diff s}_{=0} + \frac{1}{2}\int_0^\infty\mu'_{\varepsilon}(s)\left\|\Theta^t(s)\right\|^2_{\mathbb{X}^2}\diff s  \notag \\ 
& \le -\frac{\delta}{2}\int_0^\infty\mu_{\varepsilon}(s)\left\|\Theta^t(s)\right\|^2_{\mathbb{X}^2}\diff s  \notag \\
& = -\frac{\delta}{2}\left\|\Theta^t\right\|^2_{\mathcal{M}^0_\ep}.  \notag
\end{align}

Multiplying the nonlinear term by $Z$ in $\mathbb{X}^2$ produces, with an application of integration by parts, 
\begin{align}
\left\langle F(U),Z\right\rangle _{\mathbb{X}^{2}}& =\int_{\Omega
}f\left( u\right) \left( -\Delta u + \alpha u \right)
\diff x+\int_{\Gamma }\widetilde{g}\left( u\right) \left( -\Delta
_{\Gamma }u+\partial_{\mathbf{n}}u+\beta u\right) \diff\sigma 
\label{qest5bis} \\
& =\int_{\Omega }f'\left( u\right) \left\vert \nabla
u\right\vert ^{2}\diff x+\int_{\Gamma }\widetilde{g}'\left(
u\right) \left\vert \nabla _{\Gamma }u\right\vert ^{2}\diff\sigma  
\notag \\
& +\alpha\int_\Omega f(u)u \diff x+\beta\int_{\Gamma }\widetilde{g}\left( u\right) u \diff\sigma  \notag \\
& +\int_{\Gamma }\left( \widetilde{g}\left( u\right) -f\left(
u\right) \right) \partial_{\mathbf{n}}u \diff\sigma.  \notag
\end{align}
Directly from (\ref{assm-3}) and \eqref{assm-4}, we see that there holds,
\begin{align}
\int_{\Omega }f'\left( u\right) \left\vert \nabla
u\right\vert ^{2}\diff x+\int_{\Gamma }\widetilde{g}'\left(
u\right) \left\vert \nabla _{\Gamma }u\right\vert ^{2}\diff\sigma \ge -M_f\|\nabla u\|^2_{L^2(\Omega)}-M_g\|\nabla_\Gamma u\|^2_{L^2(\Gamma)},  \label{qest6}
\end{align}
and from \eqref{cons-1}-\eqref{cons-2}, we obtain,
\begin{align}
& \alpha\int_\Omega f(u)u \diff x + \beta\int_{\Gamma }\widetilde{g}\left( u\right) u \diff\sigma  \label{qest7} \\
& = \alpha\int_\Omega f(u)u \diff s + \beta\int_{\Gamma}g(u)u \diff\sigma - \int_\Gamma \omega\beta^2 u^2 \diff\sigma   \notag \\
& \ge -\alpha\kappa_1\|u\|^2_{L^2(\Omega)} - \alpha\kappa_2-\beta\kappa_3\|u\|^2_{L^2(\Gamma)}-\beta\kappa_4-\omega\beta^2\|u\|^2_{L^2(\Gamma)}  \notag \\
& \ge -C\left( \left\| U \right\|^2_{\mathbb{X}^2}+1 \right),  \notag
\end{align}
for some constant $C>0$, independent of $t.$
For the last term in \eqref{qest5bis} we recall \cite[Proof of Theorem 3.11]{Gal-Shomberg15-2}.
Due to the assumptions \eqref{assm-1}-\eqref{assm-2} it suffices to bound boundary integrals of the form, for some $r<\frac{5}{2}$,
\begin{equation*}
I:=\int_\Gamma u^{r+1}\partial_{\mathbf{n}} u \diff\sigma.
\end{equation*}
Indeed, thanks to the trace and regularity embeddings, for all $\omega\in(0,1)$ and for some $C_\omega\sim\frac{C}{\omega}>0$,
\begin{align}
I & \le \|\partial_{\mathbf{n}}u\|_{H^{1/2}(\Gamma)}\|u^{r+1}\|_{H^{-1/2}(\Gamma)}  \notag \\ 
& \le \frac{\omega}{4}\|u\|^2_{H^{2}(\Omega)} + C_\omega\|u^{r+1}\|^2_{H^{-1/2}(\Gamma)}.  \label{qest8} 
\end{align}
To bound the last term in \eqref{qest8} we will employ the Sobolev embeddings (recall $\Gamma$ is two-dimensional) $H^{1/2}(\Gamma) \hookrightarrow L^{4}(\Gamma)$ and $H^{1}(\Gamma) \hookrightarrow L^{s}(\Gamma),$ for any $s\in (\frac{4}{3},\infty )$. 
Then, by employing some basic H\"{o}lder inequalities
\begin{align}
\left\Vert u^{r+1}\right\Vert _{H^{-1/2}\left( \Gamma \right) }^{2} & = \sup_{\psi \in H^{1/2}\left( \Gamma \right) :\left\Vert \psi \right\Vert
_{H^{1/2}\left( \Gamma \right) }=1}\left\vert \left\langle u^{r+1},\psi
\right\rangle_{L^2(\Gamma)} \right\vert ^{2}  \notag \\
& \le \sup_{\psi \in H^{1/2}\left( \Gamma \right) :\left\Vert \psi \right\Vert
_{H^{1/2}\left( \Gamma \right) }=1}\left\|u^{r+1}\psi \right\|_{L^1(\Gamma)}^{2}  \notag \\
& \leq \left\Vert u\right\Vert _{L^{s}\left( \Gamma \right)
}^{2}\left\Vert u\right\Vert _{L^{\overline{s}r}\left( \Gamma \right)
}^{2r}  \notag \\
& \leq C\left\Vert u\right\Vert _{H^{1}\left( \Gamma \right)
}^{2}\left\Vert u\right\Vert _{L^{\overline{s}r}\left( \Gamma \right)
}^{2r},  \label{qest5qqq}
\end{align}
for some positive constant $C$ and for sufficiently
large $s\in (\frac{4}{3},\infty )$, where $\overline{s}:=4s/\left(
3s-4\right) >4/3$. 
Next we exploit the interpolation inequality
\begin{equation*}
\left\Vert u\right\Vert _{L^{\overline{s}r}\left( \Gamma \right) }\leq
C\left\Vert u\right\Vert _{H^{2}\left( \Gamma \right) }^{1/2r} \left\Vert u\right\Vert _{L^{2}\left( \Gamma \right) }^{1-1/2r},
\end{equation*}%
provided that $r=1+2/\overline{s}<5/2$, where we further infer from (\ref{qest5qqq}) that
\begin{align}
\left\Vert u^{r+1}\right\Vert _{H^{-1/2}\left( \Gamma \right) }^{2} &
\leq C\left\Vert u\right\Vert _{H^{1}\left( \Gamma \right)
}^{2}\left\Vert u \right\Vert _{H^{2}\left( \Gamma \right) }\left\Vert
u \right\Vert _{L^{2}\left( \Gamma \right) }^{2r-1}  \notag \\ 
& \leq \eta \left\Vert u \right\Vert _{H^{2}\left( \Gamma \right)
}^{2}+C_{\eta }\left\Vert u \right\Vert _{H^{1}\left( \Gamma \right)
}^{2}\left( \left\Vert u \right\Vert _{H^{1}\left( \Gamma \right)
}^{2}\left\Vert u \right\Vert _{L^{2}\left( \Gamma \right) }^{2\left(
2r-1\right) }\right),  \label{qest5qq}
\end{align}
for any $\eta \in (0,1]$. 
Inserting (\ref{qest5qq}) into (\ref{qest8}) and choosing a sufficiently small $\eta =\omega /C_{\omega}$, by virtue of (\ref{equiv}), we easily deduce
\begin{equation}
I \leq \frac{\omega}{4} \left\Vert Z\right\Vert _{\mathbb{X}^{2}}^{2} + C_{\omega}\left\Vert u \right\Vert _{H^{1}\left( \Gamma \right)
}^{2}\left( \left\Vert u \right\Vert _{H^{1}\left( \Gamma \right)
}^{2}\left\Vert u \right\Vert _{L^{2}\left( \Gamma \right) }^{2\left(
2r-1\right) }\right).  \label{qest5last}
\end{equation}
Together, (\ref{qest6})-(\ref{qest5last}) provide the following bound on (\ref{qest5bis}) for all $\omega>0$, and for some positive constants $C$ and $C_{\omega}\sim\frac{C}{\omega}$,
\begin{align}
\left\langle F(U),Z\right\rangle _{\mathbb{X}^{2}} \ge & - C\left(\|U\|^2_{\mathbb{X}^2}+1\right) - \frac{\omega}{4}\|Z\|^2_{\mathbb{X}^2}  \notag \\ 
& - C_{\omega}\left\Vert u \right\Vert _{H^{1}\left( \Gamma \right)
}^{2}\left( \left\Vert u \right\Vert _{H^{1}\left( \Gamma \right)
}^{2}\left\Vert u \right\Vert _{L^{2}\left( \Gamma \right) }^{2\left(
2r-1\right) }\right).  \label{qest9}
\end{align}
Also with Young's inequality,
\begin{align}
\omega\alpha\left\langle u,z \right\rangle_{L^2(\Omega)} & \le \omega\alpha^2\|u\|^2_{L^2(\Omega)} + \frac{\omega}{4}\|z\|^2_{L^2(\Omega)} \label{q2-1} \\
& \le \omega\alpha^2\|u\|^2_{L^2(\Omega)} + \frac{\omega}{4}\|Z\|^2_{\mathbb{X}^2}.  \notag
\end{align}

Applying the estimates (\ref{qest14}), (\ref{qest9}) and (\ref{q2-1}) to (\ref{qest13}), we arrive at the differential inequality, which holds for almost all $t\ge0,$ and for $0<r<\frac{5}{2},$
\begin{align}
\frac{\diff}{\diff t}&\left\{ \left\| U \right\|^2_{\mathbb{V}^1}+\left\| \Theta^t \right\|^2_{\mathcal{M}^0_\ep} \right\}  \notag \\ 
&+ \omega\|Z\|^2_{\mathbb{X}^2} + \delta\left\|\Theta^t \right\|^2_{\mathcal{M}^0_\ep}  \label{cabsball-1} \\ 
&\le C\left( \left\| U \right\|^2_{\mathbb{X}^2}+1\right) + C_\omega \|u\|^2_{H^1(\Gamma)}\left( \left\| u \right\|^2_{H^1(\Gamma)}\left\| u \right\|^{2(2r-1)}_{L^2(\Gamma)} \right).  \notag 
\end{align}
On the left-hand side, we estimate the term $\omega\|Z\|^2_{\mathbb{X}^2}$ using 
\begin{align}
\left\|U\right\|^2_{\mathbb{V}^1}&=\left\langle U,{\rm{A_W^{\alpha,\beta}}}U \right\rangle_{\mathbb{X}^2}  \notag \\ 
&=\left\langle U,Z\right\rangle_{\mathbb{X}^2}  \notag \\ 
& \le C_\omega\left\|U\right\|^2_{\mathbb{X}^2}+\omega\left\|Z\right\|^2_{\mathbb{X}^2}.  \label{cabsball-5}
\end{align}
Finally, with (\ref{cabsball-5}) and the uniform bounds (\ref{weak-bound}), we now obtain from (\ref{cabsball-1}), with $m_1:=\min\{1,\delta\}>0$,
\begin{align}
\frac{\diff}{\diff t}&\left\{ \left\| U \right\|^2_{\mathbb{V}^1}+\left\| \Theta^t \right\|^2_{\mathcal{M}^0_\ep} \right\}   \label{cabsball-2} \\
&  + m_1\left( \|U\|^2_{\mathbb{V}^1} + \left\|\Theta^t \right\|^2_{\mathcal{M}^0_\ep} \right)  \notag \\ 
& \le C_\omega \left( 1+\|u\|^2_{H^1(\Gamma)} \right) \left( \|U\|^2_{\mathbb{V}^1} + \|\Theta^t\|^2_{\mathcal{M}^0_\ep} \right) + C,  \notag
\end{align}
where $C_\omega>0$ depends on $\widetilde{P}_0$ from (\ref{weak-bound}).
Now from \eqref{fus-6}, we immediately find the following dissipation integral
\begin{align}
\omega\int_t^{t+1} \|U(\tau)\|^2_{\mathbb{V}^1} \diff\tau \le C,
\end{align}
and we may apply a Gr\"{o}nwall-type inequality (see e.g. Proposition \ref{GL} below) to (\ref{cabsball-2}).
We also recall (\ref{equiv}) yields, for some $C_*>0,$
\begin{align}
C_*^{-1}\left\| \Phi^t \right\|^2_{\mathbb{V}^2} \le \left\| \mathrm{A_{W}^{\alpha,\beta}}\Phi^t \right\|^2_{\mathcal{M}^0_\ep} = \left\|\Theta^t\right\|^2_{\mathcal{M}^0_\ep} \le C_*\left\| \Phi^t \right\|^2_{\mathbb{V}^2}.  \label{eqivi2}
\end{align}
Hence, there are constants $M_1\ge 1$ and $P_1>0$, both uniform in $t$, such that for all $t\ge0$, (\ref{cabsball-2}) produces, for all $t\ge0,$
\begin{align}
& \left\| U(t) \right\|^2_{\mathbb{V}^1} + \left\| \Phi^t \right\|^2_{\mathcal{M}^2_{\varepsilon}}  \notag \\
& \le M_1 e^{-m_1t}\left( \left\| U_0 \right\|^2_{\mathbb{V}^1} + \left\| \Phi_0 \right\|^2_{\mathcal{M}^2_{\varepsilon}} \right) + P_1  \notag \\ 
& \le M_1 R e^{-m_1t} + P_1,  \label{cabsball-9}
\end{align}
where the last inequality follows because $\|\Phi_0\|_{\mathcal{M}^2_{\varepsilon}}\le\|\Phi_0\|_{\mathcal{K}^2_{\varepsilon}}\le R$.

To show (\ref{strong-decay}) holds we need to control the last two terms of the norm \eqref{new-norm}.
First, it is easy to see from \eqref{cabsball-9} that for all $t\ge0$ 
\begin{align}
\left\| U(t) \right\|^2_{\mathbb{V}^1} \le \left\| U(t) \right\|^2_{\mathbb{V}^1} + \left\| \Phi^t \right\|^2_{\mathcal{M}^2_{\varepsilon}} \le M_1R+P_1.  \notag
\end{align}
Then the conclusions of Lemmas \ref{what-1} and \ref{what-2} given above now take the form 
\begin{align}
& \ep\|{\rm{T}}_\ep\Phi^t\|^2_{\mathcal{M}^0_\ep} + \sup_{\tau\ge0} \tau \mathbb{T}_\ep(\tau;\Phi^t)  \notag \\
& \le e^{-\delta t} \left( \ep \|{\rm{T}}_\ep\Phi_0\|^2_{\mathcal{M}^0_\ep} + 2\sup_{\tau\ge1} \tau\mathbb{T}_\ep(\tau;\Phi_0) \left( t+2 \right) \right) + M_1Re^{-m_1t}+P_1  \notag \\ 
& \le e^{-m_1t} \left( R(M_1+1) + Q(R) \left( t+1 \right) \right) + P_1  \notag \\ 
& \le Q(R) e^{-m_1t} \left( t+1 \right) + P_1.  \label{strong-bound-1} 
\end{align}
Together, the estimates \eqref{cabsball-9} and \eqref{strong-bound-1} show that \eqref{strong-decay} holds.

The existence of a bounded set $\mathcal{B}^1_\ep$ in $\mathcal{V}^1_{\varepsilon}$ that is absorbing and positively invariant for $\mathcal{S}_\ep(t)$ follows from \eqref{strong-decay}.
Indeed, define
\begin{equation*}
\mathcal{B}^1_{\varepsilon}:=\left\{ (U,\Phi)\in \mathcal{V}^1_{\varepsilon} : \left\| (U,\Phi) \right\|_{\mathcal{V}^1_{\varepsilon}}\le \sqrt{2P_1+1} \right\}.
\end{equation*}
Then, given any nonempty bounded subset $B$ in $\mathcal{H}^0_{\varepsilon}\setminus\mathcal{B}^1_{\varepsilon}$, and after possibly enlarging the radius of $\mathcal{B}^1_{\varepsilon}$ in $\mathcal{H}^0_{\varepsilon}$ due to the embedding $\mathcal{V}^1_{\varepsilon}\hookrightarrow\mathcal{H}^0_{\varepsilon}$, we have that $\mathcal{S}_{\varepsilon}(t)B\subseteq\mathcal{B}^1_{\varepsilon}$, in $\mathcal{H}^0_{\varepsilon}$, for all $t\ge t_1$ where $t_1=t_1(R)\ge0$ is such that there holds
\begin{equation}
e^{-\min\{\delta,1\}t_1}\left( t_1+1 \right) \le \frac{1}{Q(R)}.  \label{time1}
\end{equation}
This establishes (C1) and completes the proof when $\ep\in(0,1]$.
\end{proof}

The following result refers to the strong solutions developed in \cite[Theorem 3.11]{Gal-Shomberg15-2} (see Theorem \ref{t:strong-solutions} above) whose initial data is now taken in $\mathcal{V}^1_\ep\subset\mathcal{H}^0_\ep$.

\begin{corollary}
For all $\Upsilon=(U_0,\Phi_0)\in\mathcal{H}^1_{\varepsilon}=\mathbb{V}^1\times\mathcal{M}^2_\varepsilon$, it follows that any strong solution $(U(t),\Phi^t)$ to Problem {\textbf{P}}$_{\varepsilon}$ is bounded, uniformly in $t$ and $\ep$; indeed, thanks to \eqref{strong-decay} there is a constant $\widetilde{P}_1>0$, depending on the bound $P_1$ and the initial datum, but independent of $t$ and $\ep$, in which, 
\begin{equation} \label{strong-bound}
\limsup_{t\rightarrow+\infty} \left\|\mathcal{S}_\varepsilon(t)(U_0,\Phi_0)\right\|_{\mathcal{V}^1_{\varepsilon}} \le \widetilde{P}_1.
\end{equation}
\end{corollary}

We can now give a decay estimate for $\Phi^t$ in $\mathcal{M}^1_\varepsilon$.

\begin{lemma}\label{t:Phi-decay-1}
There holds, for all $\varepsilon\in(0,1]$, $\omega\in(0,1)$, $\Upsilon_0=(U_0,\Phi_0)\in\mathcal{V}^1_\varepsilon$, and for all $t\geq 0$,
\begin{equation}
\label{Phi-decay-1}
\left\| \Phi^t \right\|^2_{\mathcal{M}^1_\varepsilon} \leq \left\| \Phi_0 \right\|^2_{\mathcal{M}^1_\varepsilon} e^{-\delta t/2\varepsilon} + C(\widetilde{P}_0)\varepsilon.
\end{equation}
\end{lemma}

\begin{proof}
Let $\varepsilon\in(0,1]$, $\omega\in(0,1)$ and $\Upsilon_0=(U_0,\Phi_0)\in\mathcal{H}^0_\varepsilon$.
As in the proof of Lemma \ref{weak-ball}, take $\Pi=\Phi^t(s)$ in equation (\ref{weak-solution-2}) to obtain 
\begin{equation*}\begin{aligned}
\int_0^\infty & \mu_\varepsilon(s) \left\langle \partial_t \Phi^t(s),{\rm{A_W^{\alpha,\beta}}} \Phi^t(s) \right\rangle_{\mathbb{X}^2} \diff s \\ 
& = \int_0^\infty \mu_\varepsilon(s) \left\langle {\rm{T}}_\varepsilon \Phi^t(s),{\rm{A_W^{\alpha,\beta}}} \Phi^t(s) \right\rangle_{\mathbb{X}^2} \diff s + \int_0^\infty \mu_\varepsilon(s) \left\langle U,{\rm{A_W^{\alpha,\beta}}} \Phi^t(s) \right\rangle_{\mathbb{X}^2} \diff s.
\end{aligned}\end{equation*}
Combining \eqref{operator-T-1}, \eqref{fus-2}, \eqref{fus-5}, and \eqref{fus-7}, we obtain
\begin{align}
\frac{1}{2}\frac{\diff}{\diff t} \left\| \Phi^t \right\|^2_{\mathcal{M}^1_\varepsilon} + \frac{\delta}{2\varepsilon}\left\|\Phi^t\right\|^2_{\mathcal{M}^1_\varepsilon} \le \left\langle U,\Phi^t \right\rangle_{\mathcal{M}^1_\varepsilon}.  \label{Phi-decay-3}
\end{align}
Estimating the product on the right-hand side with Young's inequality,
\begin{equation}\begin{aligned}
\label{Phi-decay-5}
\left\langle U,\Phi^t \right\rangle_{\mathcal{M}^1_\varepsilon} & = \int_0^\infty \mu_\varepsilon(s) \left\langle U,\Phi^t \right\rangle_{\mathbb{V}^1} \diff s \\ 
& \le \int_0^\infty \mu_\varepsilon(s) \left\| U \right\|_{\mathbb{V}^1} \left\|\Phi^t \right\|_{\mathbb{V}^1} \diff s \\ 
& \leq \left\|U\right\|_{\mathbb{V}^1}\left\|\Phi^t\right\|_{\mathcal{M}^1_\varepsilon} \\
& \leq \frac{1}{\delta}\left\|U\right\|^2_{\mathbb{V}^1} + \frac{\delta}{4\varepsilon}\left\| \Phi^t \right\|^2_{\mathcal{M}^1_\varepsilon},
\end{aligned}\end{equation}
we combine (\ref{Phi-decay-3}) and (\ref{Phi-decay-5}) to find that, for almost all $t\geq 0$,
\begin{equation}
\label{Phi-decay-6}
\frac{\diff}{\diff t}\left\| \Phi^t \right\|^2_{\mathcal{M}^1_\varepsilon} + \frac{\delta}{4\varepsilon} \left\| \Phi^t \right\|^2_{\mathcal{M}^1_\varepsilon} \leq \frac{1}{\delta}\|U\|^2_{\mathbb{V}^1}.
\end{equation}
Thus, applying a Gr\"{o}nwall type inequality whereby integrating (\ref{Phi-decay-6}) over the interval $(0,t)$, recalling the uniform bound (\ref{strong-bound}), produces (\ref{Phi-decay-1}).
\end{proof}

\begin{corollary}
From Lemma \ref{t:Phi-decay-1} we obtain the limit, for each $t>0$ fixed, 
\begin{equation}
\label{Phi-decay-2}
\lim_{\varepsilon\rightarrow 0}\left\| \Phi^t \right\|_{\mathcal{M}^1_\varepsilon} = 0.
\end{equation}
In addition, since $e^{-\delta t/2\varepsilon} < e^{-\delta t/2} \varepsilon^{\delta t/2} < \varepsilon^{\delta T/2}$ for all $\varepsilon\in(0,1]$ and for all $t$ in the compact interval $[0,T]$, for some $T>0$, then inequality (\ref{Phi-decay-1}) is estimated by,
\[
\left\| \Phi^t \right\|^2_{\mathcal{M}^1_\varepsilon} \leq \max\left\{ \left\| \Phi_0 \right\|^2_{\mathcal{M}^1_\varepsilon},C(\widetilde{P}_0) \right\} \left( e^{\delta T/2}+\varepsilon \right). 
\]
Define the constants $\Lambda_0 = \max\left\{ \left\| \Phi_0 \right\|^2_{\mathcal{M}^1_\varepsilon} e^{-\delta T/2},C(\widetilde{P}_0) \right\}^{1/2}$ and $p_0=\min\left\{ \frac{\delta T}{4},\frac{1}{2} \right\}$.
Then, for all $\varepsilon\in(0,1]$ and for all $t\in[0,T]$, there holds, 
\[
\left\| \Phi^t \right\|_{\mathcal{M}^1_\varepsilon} \leq \Lambda_0 \varepsilon^{p_0}.
\]
\end{corollary}

We now go on to establish the next condition of Proposition \ref{abstract1}.

\begin{lemma}  \label{t:to-C2} 
Condition (C2) holds for each $\ep\in(0,1]$ and $\omega\in(0,1)$. 
The constants $t^*$ and $\ell^*$ depend on $\omega, \delta$ and the constant due to the embedding $\mathbb{V}^1\hookrightarrow \mathbb{X}^2.$
\end{lemma}

\begin{proof}
Let $\ep\in(0,1]$ and $\omega\in(0,1)$. 
Let $\Upsilon_0=(U_0,\Phi_0), \Xi_0=(V_0,\Psi_0)\in \mathcal{B}^{1}_{\varepsilon}$. 
Define the pair of trajectories, for $t\geq 0$, $\Upsilon(t)=\mathcal{S}_{\varepsilon}(t)\Upsilon _{0}=(U(t),\Phi^t)$ and $\Xi(t)=\mathcal{S}_{\varepsilon}(t)\Xi _{0}=(V(t),\Psi^t)$. 
For each $t\geq 0$, decompose the difference $\overline{\Delta}(t):=\Upsilon (t)-\Xi(t)$ with $\overline{\Delta}_{0}:=\Upsilon_{0}-\Xi _{0}$ as follows:
\begin{equation*}
\overline{\Delta}(t)=\widehat{\Upsilon}(t)+\widehat{\Xi}(t)
\end{equation*}
where $\widehat{\Upsilon}(t)=(\widehat{V}(t),\widehat{\Psi}^t)$ and $\widehat{\Xi}(t)=(\widehat{W}(t),\widehat{\Theta}^t)$ are solutions of the problems: 
\begin{equation}
\label{diff-decomp-u}
\left\{ \begin{array}{l}
\partial_t \widehat{V}(t) + \omega {\mathrm{A^{0,\beta}_{W}}}\widehat{V}(t) 
+ \displaystyle\int_0^\infty \mu_{\varepsilon}(s) {\mathrm{A^{\alpha,\beta}_{W}}} \widehat{\Psi}^t(s) \diff s = 0, \\ 
\partial_t\widehat{\Psi}^t(s) = {\rm{T}_{\varepsilon}} \widehat{\Psi}^t(s) + \widehat{V}(t), \\
\widehat{\Upsilon}(0) = \Upsilon_0-\Xi_0,
\end{array} \right.
\end{equation}
and
\begin{equation}
\label{diff-decomp-v}
\left\{ \begin{array}{l}
\partial_t \widehat{W}(t) + \omega {\mathrm{A^{0,\beta}_{W}}}\widehat{W}(t) + \displaystyle\int_0^\infty \mu_{\varepsilon}(s) {\mathrm{A^{\alpha,\beta}_{W}}} \widehat{\Theta}^t(s) \diff s + F(U(t)) - F(V(t)) = 0, \\ 
\partial_t \widehat{\Theta}^t(s) = {\rm{T}_{\varepsilon}} \widehat{\Theta}^t(s) + \widehat{W}(t), \\
\widehat{\Xi}(0) = {\bf{0}}.
\end{array} \right.
\end{equation}

{\underline{Step 1}}. (Proof of (\ref{difference-decomposition-L}).)
By estimating along the usual lines, after multiplying (\ref{diff-decomp-u})$_1$ by $\widehat{V}$ in $\mathbb{X}^2$ and multiplying equation (\ref{diff-decomp-u})$_2$ by ${\mathrm{A^{\alpha,\beta}_{W}}}\widehat{\Psi}^t$ in $\mathcal{M}^0_\ep=L^2_{\mu_{\varepsilon}}(\mathbb{R}_+;\mathbb{X}^2)$, we easily obtain the
differential inequality,
\begin{align}
\frac{1}{2}\frac{\diff}{\diff t} & \left\{ \left\| \widehat{V} \right\|^2_{\mathbb{X}^2} + \left\| \widehat{\Psi}^t \right\|^2_{\mathcal{M}^1_{\varepsilon}} \right\} + C^{-1}_{\overline{\Omega}}\omega\left\| \widehat{V} \right\|^2_{\mathbb{X}^2} + \frac{\delta}{2}\left\| \widehat{\Psi}^t \right\|^2_{\mathcal{M}^1_{\varepsilon}} \le 0,   \label{diff-decomp-1}
\end{align}
where the constant $C_{\overline{\Omega}}>0$ is due to the embedding $\mathbb{V}^1\hookrightarrow \mathbb{X}^2$; i.e., $\|\widehat{V}\|^2_{\mathbb{X}^2} \leq C_{\overline{\Omega}}\|\widehat{V}\|^2_{\mathbb{V}^1}$.
Set $m_2:=\min\{ 2C^{-1}_{\overline{\Omega}}\omega,\delta \} > 0$.
Thus, (\ref{diff-decomp-1}) becomes, for almost all $t\geq 0$,
\begin{equation*}
\frac{\diff}{\diff t}\left\{ \left\| \widehat{V} \right\|^2_{\mathbb{X}^2} + \left\| \widehat{\Psi}^t \right\|^2_{\mathcal{M}^1_{\varepsilon}} \right\} + m_2 \left( \left\| \widehat{V} \right\|^2_{\mathbb{X}^2} + \left\| \widehat{\Psi}^t \right\|^2_{\mathcal{M}^1_{\varepsilon}} \right) \leq 0.
\end{equation*}
After applying a Gr\"{o}nwall inequality, we have that for all $t\geq 0$,
\begin{equation}  \label{to-C2-L}
\left\| \left( \widehat{V}(t),\widehat{\Psi}^t \right) \right\|_{\mathcal{H}^0_{\varepsilon}} \leq \left\|\overline{\Delta}_0 \right\|_{\mathcal{H}^0_{\varepsilon}} e^{-m_2 t/2}.
\end{equation}
Set $t^{\ast }:=\max \{t_{1},\frac{2}{m_2}\ln 4 \}$ (recall $t_1$ was defined in (\ref{time1}) in the proof of Lemma \ref{t:to-C1}).
Then, for all $t\geq t^{\ast }$, (\ref{difference-decomposition-L}) holds
with $L=\widehat{\Upsilon}(t^\ast)=(\widehat{V}(t^{\ast}),\widehat{\Phi}^{t^{\ast}}),$ and 
\begin{equation*}
\ell^{\ast } = e^{-m_2 t^{\ast }/2} < \frac{1}{2}.
\end{equation*}

Before we show that (\ref{difference-decomposition-K}) holds, we need to establish a crucial bound. 

{\underline{Step 2}}. (A preliminary bound for $\widehat{W}$ and $\widehat{\Theta}^t$.)
We claim, for each $0<T<\infty$, there holds
\begin{align}
& \widehat{W}\in L^\infty\left([0,T];\mathbb{X}^2\right)\cap L^2\left([0,T];\mathbb{V}^1\right),  \label{hat-claim-1} \\
& \widehat{\Theta}^t\in L^\infty\left([0,T];\mathcal{M}^1_{\varepsilon}\right).  \label{hat-claim-2}
\end{align}

To show this, we multiply equation (\ref{diff-decomp-v})$_1$ by $\widehat{W}$ in $\mathbb{X}^2$ and multiply equation (\ref{diff-decomp-v})$_2$ by $\mathrm{A_{W}^{\alpha,\beta}}\widehat{\Theta}^t$ in $\mathcal{M}^0_\ep$. 
Summing the resulting two identities produces,
\begin{align}
& \frac{1}{2}\frac{\diff}{\diff t} \left\{ \left\| \widehat{W} \right\|^2_{\mathbb{X}^2} + \left\| \widehat{\Theta}^t \right\|^2_{\mathcal{M}^1_{\varepsilon}} \right\} + \omega\left\langle \mathrm{A_{W}^{0,\beta}}\widehat{W},\widehat{W} \right\rangle_{\mathbb{X}^2} - \left\langle \mathrm{T}_{\varepsilon}\widehat{\Theta}^t,\widehat{\Theta}^t \right\rangle_{\mathcal{M}^1_{\varepsilon}}  \label{ba-vid-4} \\
& + \left\langle F(U)-F(V),\widehat{W} \right\rangle_{\mathbb{X}^2}  \notag \\
& = 0.  \notag
\end{align}
The first of the three products above can be re-written, using the definition of the $\mathbb{V}^1$ norm (see (\ref{v1b})), as
\begin{align}
\omega\left\langle \mathrm{A_{W}^{0,\beta}}\widehat{W},\widehat{W} \right\rangle_{\mathbb{X}^2} & = \omega\left\langle \mathrm{A_{W}^{\alpha,\beta}}\widehat{W},\widehat{W} \right\rangle_{\mathbb{X}^2} - \omega\alpha\langle \widehat{w},\widehat{w} \rangle_{L^2(\Omega)}  \label{ba-pop-2} \\ 
& = \omega\left\|\widehat{W}\right\|^2_{\mathbb{V}^1} - \omega\alpha \left\| \widehat{w} \right\|^2_{L^2(\Omega)}  \notag \\
& = \omega \left( \left\| \nabla \widehat{w} \right\|^2_{L^2(\Omega)} + \left\| \nabla_\Gamma \widehat{w} \right\|^2_{L^2(\Gamma)} + \beta\left\| \widehat{w} \right\|^2_{L^2(\Gamma)} \right).  \notag
\end{align}
As with the above estimate (\ref{qest14}), we have 
\begin{align}
\left\langle  \mathrm{T}_{\varepsilon}\widehat{\Theta}^t,\widehat{\Theta}^t \right\rangle_{\mathcal{M}^1_{\varepsilon}} \leq -\frac{\delta}{2}\left\| \widehat{\Theta}^t \right\|^2_{\mathcal{M}^1_{\varepsilon}}.  \label{ba-pop-4} 
\end{align}
Using assumptions \eqref{assm-1} and \eqref{assm-2} with data in the bounded set $\mathcal{B}^1_\ep$ and the uniform bound \eqref{weak-bound}, we now estimate the nonlinear terms as follows
\begin{align}
\langle f(u)-f(v),\widehat{w}\rangle_{L^2(\Omega)} & \le \|(f(u)-f(v))\widehat{w}\|_{L^1(\Omega)}  \notag \\ 
& \le \|f(u)-f(v)\|_{L^{6/5}(\Omega)}\|\widehat{w}\|_{L^6(\Omega)}  \notag \\ 
& \le \ell_1\|(u-v)(1+|u-v|^{r_1})\|_{L^{6/5}(\Omega)}\|\widehat{w}\|_{L^6(\Omega)}  \notag \\ 
& \le \ell_1\|u-v\|_{L^{6}(\Omega)}\left(1+\|u-v\|^{r_1}_{L^{3r_1/2}(\Omega)}\right)\|\widehat{w}\|_{L^6(\Omega)}  \notag \\ 
& \le C\|\widehat{w}\|_{H^1(\Omega)},  \label{func-1}
\end{align}
where $C=C(\ell_1,\Omega,\widetilde{P}_0,r_1)>0$ and the last inequality follows from the fact that $H^1(\Omega)\hookrightarrow L^6(\Omega)$ and $H^1(\Omega)\hookrightarrow L^{3r_1/2}(\Omega)$ because $1\le r_1<\frac{5}{2}.$
Similarly for $\widetilde{g}$ (here the estimate is easier because $H^1(\Gamma)\hookrightarrow L^p(\Gamma)$ for $1\le p<\infty$ as $\Gamma$ is two dimensional), 
\begin{align}
\langle \widetilde{g}(u)-\widetilde{g}(v),\widehat{w}\rangle_{L^2(\Gamma)} & \le C\|\widehat{w}\|_{H^1(\Gamma)}.  \label{func-2}
\end{align}
Thus, \eqref{func-1} and \eqref{func-2} show that
\begin{align}
\left|\left\langle F(U)-F(V),\widehat{W} \right\rangle_{\mathbb{X}^2}\right| & \le C_\omega\left\| {\overline{\Delta}}_0 \right\|^2_{\mathcal{H}^0_{\varepsilon}} + \frac{\omega}{2}\left\| \widehat{W} \right\|^2_{\mathbb{V}^1},  \label{ba-pop-5}
\end{align}
where $C_\omega\sim\frac{C}{\omega}$.
Together (\ref{ba-vid-4})-(\ref{ba-pop-5}) yields the differential inequality, which holds for almost all $t\ge0$,
\begin{align}
& \frac{\diff}{\diff t} \left\{ \left\| \widehat{W} \right\|^2_{\mathbb{X}^2} + \left\| \widehat{\Theta}^t \right\|^2_{\mathcal{M}^1_{\varepsilon}} \right\}   \label{ba-pop-9} \\
& + \omega \left( \left\| \nabla \widehat{w} \right\|^2_{L^2(\Omega)} + \left\| \nabla_\Gamma \widehat{w} \right\|^2_{L^2(\Gamma)} + \beta\left\| \widehat{w} \right\|^2_{L^2(\Gamma)} \right)  + \delta\left\| \widehat{\Theta}^t \right\|^2_{\mathcal{M}^1_{\varepsilon}}  \notag \\
& \le C_{\omega}\left\| {\overline{\Delta}}_0 \right\|^2_{\mathcal{H}^0_{\varepsilon}}.  \notag
\end{align}
Now integrating (\ref{ba-pop-9}) with respect to $t$ in $[0,T]$, for some fixed $0<T<\infty$, we obtain
\begin{align}
& \left\| \widehat{W}(t) \right\|^2_{\mathbb{X}^2} + \left\| \widehat{\Theta}^t \right\|^2_{\mathcal{M}^1_{\varepsilon}}   \label{ba-pop-10} \\
&+ \int_0^t \left( \omega \left( \left\| \nabla \widehat{w}(\tau) \right\|^2_{L^2(\Omega)} + \left\| \nabla_\Gamma \widehat{w}(\tau) \right\|^2_{L^2(\Gamma)} + \beta\left\| \widehat{w}(\tau) \right\|^2_{L^2(\Gamma)} \right) + \delta\left\| \widehat{\Theta}^\tau \right\|^2_{\mathcal{M}^1_{\varepsilon}} \right) \diff\tau  \notag \\ 
& \le C_{\omega}\left\| {\overline{\Delta}}_0 \right\|^2_{\mathcal{H}^0_{\varepsilon}}T.  \notag
\end{align}
Using (\ref{ba-pop-10}), we easily deduce the claim \eqref{hat-claim-1}-\eqref{hat-claim-2}.

{\underline{Step 3}}. (Proof of (\ref{difference-decomposition-K}).) 
We begin by multiplying equation (\ref{diff-decomp-v})$_1$ by $K=\mathrm{A_{W}^{\alpha,\beta}}\widehat{W}$ in $\mathbb{X}^2$, then, after applying $\mathrm{A_{W}^{\alpha,\beta}}$ to equation (\ref{diff-decomp-v})$_2$, we multiply the result by $\Lambda^t=\mathrm{A_{W}^{\alpha,\beta}}\widehat{\Theta}^t$ in $\mathcal{M}^0_\ep=L^2_{\mu_{\varepsilon}}(\mathbb{R}_+;\mathbb{X}^2)$. 
This leaves us with the two identities,
\begin{align}
& \left \langle \partial_t\widehat{W},K \right\rangle_{\mathbb{X}^2} + \omega\left\langle \mathrm{A_{W}^{0,\beta}}\widehat{W},K \right\rangle_{\mathbb{X}^2} + \left\langle \mathrm{A_{W}^{\alpha,\beta}}\widehat{\Theta}^t,K \right\rangle_{\mathcal{M}^0_\ep}  \label{vid-1} \\
& + \left\langle F(U)-F(V),K \right\rangle_{\mathbb{X}^2}  \notag = 0.   \notag
\end{align}
and 
\begin{align}
\left\langle \partial_t\mathrm{A_{W}^{\alpha,\beta}}\widehat{\Theta}^t, \Lambda^t \right\rangle_{\mathcal{M}^0_\ep} = \left\langle \mathrm{A_{W}^{\alpha,\beta}} \mathrm{T}_{\varepsilon}\widehat{\Theta}^t,\Lambda^t \right\rangle_{\mathcal{M}^0_\ep} + \left\langle \mathrm{A_{W}^{\alpha,\beta}}\widehat{W},\Lambda^t \right\rangle_{\mathcal{M}^0_\ep}.  \label{vid-2} 
\end{align}
Observe, 
\begin{align}
\left\langle \mathrm{A_{W}^{\alpha,\beta}}\widehat{\Theta}^t,K \right\rangle_{\mathcal{M}^0_\ep} & = \left\langle \Lambda^t,\mathrm{A_{W}^{\alpha,\beta}}\widehat{W} \right\rangle_{\mathcal{M}^0_\ep}.  \label{vid-22} 
\end{align}
Hence, combining (\ref{vid-1}) and (\ref{vid-2}) through (\ref{vid-22}),
\begin{align}
& \left \langle \partial_t\widehat{W},K \right\rangle_{\mathbb{X}^2} + \omega\left\langle \mathrm{A_{W}^{0,\beta}}\widehat{W},K \right\rangle_{\mathbb{X}^2} + \left\langle \partial_t\mathrm{A_{W}^{\alpha,\beta}}\widehat{\Theta}^t, \Lambda^t \right\rangle_{\mathcal{M}^0_\ep} - \left\langle \mathrm{A_{W}^{\alpha,\beta}} \mathrm{T}_{\varepsilon}\widehat{\Theta}^t,\Lambda^t \right\rangle_{\mathcal{M}^0_\ep}  \label{vid-3} \\
& + \left\langle F(U)-F(V),K \right\rangle_{\mathbb{X}^2}  \notag \\
& = 0.   \notag
\end{align}
The first three products can be re-written as follows,
\begin{align}
\left \langle \partial_t\widehat{W},K \right\rangle_{\mathbb{X}^2} & = \left\langle \partial_t\widehat{W},\mathrm{A_{W}^{\alpha,\beta}}\widehat{W} \right\rangle_{\mathbb{X}^2}  \label{pop-1} \\
& = \left\langle \partial_t\widehat{W},\widehat{W} \right\rangle_{\mathbb{V}^1}  \notag \\
& = \frac{1}{2}\frac{\diff}{\diff t}\left\| \widehat{W} \right\|^2_{\mathbb{V}^1},  \notag
\end{align}
\begin{align}
\omega\left\langle \mathrm{A_{W}^{0,\beta}}\widehat{W},K \right\rangle_{\mathbb{X}^2} & = \omega\left\langle \mathrm{A_{W}^{\alpha,\beta}}\widehat{W},K \right\rangle_{\mathbb{X}^2} - \omega\alpha\langle \widehat{w},k \rangle_{L^2(\Omega)}  \label{pop-2} \\ 
& = \omega\left\|K\right\|^2_{\mathbb{X}^2} - \omega\alpha\langle \widehat{w},k \rangle_{L^2(\Omega)},  \notag
\end{align}
and
\begin{align}
\left\langle \partial_t\mathrm{A_{W}^{\alpha,\beta}}\widehat{\Theta}^t, \Lambda^t \right\rangle_{\mathcal{M}^0_\ep} & = \left\langle \partial_t\Lambda^t, \Lambda^t \right\rangle_{\mathcal{M}^0_\ep}  \label{pop-3} \\ 
& = \frac{1}{2}\frac{\diff}{\diff t}\left\| \Lambda^t \right\|^2_{\mathcal{M}^0_\ep}.  \notag
\end{align}
Inserting (\ref{pop-1})-(\ref{pop-3}) into (\ref{vid-3}) gives us the differential identity, 
\begin{align}
& \frac{1}{2}\frac{\diff}{\diff t} \left\{ \left\| \widehat{W} \right\|^2_{\mathbb{V}^1} + \left\| \Lambda^t \right\|^2_{\mathcal{M}^0_\ep} \right\} + \omega \left\| K \right\|^2_{\mathbb{X}^2} - \left\langle \mathrm{T}_{\varepsilon}\widehat{\Theta}^t,\Lambda^t \right\rangle_{\mathcal{M}^0_\ep}  \label{vid-4} \\
& + \left\langle F(U)-F(V),K \right\rangle_{\mathbb{X}^2}  \notag \\
& = \omega\alpha\left\langle \widehat{w},k \right\rangle_{L^2(\Omega)}.  \notag
\end{align}
Similar to (\ref{qest14}), we estimate
\begin{align}
\left\langle \mathrm{A_{W}^{\alpha,\beta}} \mathrm{T}_{\ep}\widehat{\Theta}^t,\Lambda^t \right\rangle_{\mathcal{M}^0_\ep} \leq -\frac{\delta}{2}\left\| \Lambda^t \right\|^2_{\mathcal{M}^0_\ep},  \label{pop-4} 
\end{align}
and in a similar fashion to \eqref{ba-pop-5}, we find
\begin{align}
\left|\left\langle F(U)-F(V),K \right\rangle_{\mathbb{X}^2}\right| & \le C_\omega\left\| {\overline{\Delta}}_0 \right\|^2_{\mathcal{H}^0_{\varepsilon}} + \frac{\omega}{2}\|K\|^2_{\mathbb{V}^1},  \label{pop-5}
\end{align}
where $C_\omega\sim\frac{C}{\omega}$.
We also estimate
\begin{align}
\omega\alpha\left\langle \widehat{w},k \right\rangle_{L^2(\Omega)} \le \omega\alpha^2\left\| \widehat{W} \right\|^2_{\mathbb{X}^2} + \frac{\omega}{4}\left\| K \right\|^2_{\mathbb{X}^2}  \label{pop-6} 
\end{align}
Together (\ref{vid-4})-(\ref{pop-6}) yields the differential inequality, which holds for almost all $t\ge0$,
\begin{align}
& \frac{\diff}{\diff t} \left\{ \left\| \widehat{W} \right\|^2_{\mathbb{V}^1} + \left\| \Lambda^t \right\|^2_{\mathcal{M}^0_\ep} \right\} + \omega \left\| \mathrm{A_{W}^{\alpha,\beta}}\widehat{W} \right\|^2_{\mathbb{X}^2} + \delta\left\| \Lambda^t \right\|^2_{\mathcal{M}^0_\ep}  \label{pop-9} \\
& \le \alpha\left\| \widehat{W} \right\|^2_{\mathbb{X}^2} + C_{\omega}\left\| {\overline{\Delta}}_0 \right\|^2_{\mathcal{H}^0_{\varepsilon}}.  \notag
\end{align}
Now integrating (\ref{pop-9}) with respect to $t$ in $[0,T]$, for some fixed $0<T<\infty$, we obtain
\begin{align}
& \left\| \widehat{W}(t) \right\|^2_{\mathbb{V}^1} + \left\| \Lambda^t \right\|^2_{\mathcal{M}^0_\ep} + \int_0^t \left( \omega \left\| \mathrm{A_{W}^{0,\beta}}\widehat{W}(\tau) \right\|^2_{\mathbb{X}^2} + \delta\left\| \Lambda^\tau \right\|^2_{\mathcal{M}^0_\ep} \right) \diff\tau  \notag \\ 
& \le \int_0^t \alpha\left\| \widehat{W}(\tau) \right\|^2_{\mathbb{X}^2} \diff\tau + C_{\omega}\left\| {\overline{\Delta}}_0 \right\|^2_{\mathcal{H}^0_{\varepsilon}}T,  \label{pop-10}
\end{align}
where the right-hand side of the inequality makes sense thanks to (\ref{hat-claim-1}).
Now omitting the second and third terms from the left-hand side of \eqref{pop-10}, the following bound follows easily with Gr\"{o}nwall's inequality
\begin{align}
\left\| \widehat{W}(t) \right\|^2_{\mathbb{V}^1} & \le C_{\omega}\left\| {\overline{\Delta}}_0 \right\|^2_{\mathcal{H}^0_{\varepsilon}}T e^{\alpha T},  \label{david-1}
\end{align}
and with this
\begin{align}
\left\|\Lambda^t\right\|^2_{\mathcal{M}^0_\ep} & \le C_{\omega}\left\| {\overline{\Delta}}_0 \right\|^2_{\mathcal{H}^0_{\varepsilon}}T e^{\alpha T},  \label{david-2}
\end{align}
also follows.

In order to obtain the desired bound from \eqref{david-1} and \eqref{david-2}, first recall that there is $C_*>0$ (cf. (\ref{equiv})) such that 
\begin{equation*}
\left\|\Lambda^t\right\|^2_{\mathcal{M}^0_\ep} = \left\|\mathrm{A_{W}^{\alpha,\beta}}\widehat{\Theta}^t\right\|^2_{\mathcal{M}^0_\ep} \ge C_*^{-1} \left\|\widehat{\Theta}^t\right\|^2_{\mathcal{M}^2_{\varepsilon}}.
\end{equation*}
Thus, letting $T=t^*$ (from Step 1), we obtain, for some positive monotonically increasing function $M_2(\cdot)$,
\begin{equation}  \label{top-4}
\left\| \left( \widehat{W}(t^*),\widehat{\Theta}^{t^*} \right) \right\| _{\mathcal{H}^1_{\varepsilon}} \le M_2(t^*) \left\| {\overline{\Delta}}_0 \right\|_{\mathcal{H}^0_{\varepsilon}}.
\end{equation}
Now it suffices to show that for some positive constant $C(T)$, there holds for all $t\in[0,T],$
\begin{align}
\left\|\widehat{\Theta}^t\right\|^2_{\mathcal{K}^2_{\varepsilon}} & \le C(T)\left\| {\overline{\Delta}}_0 \right\|^2_{\mathcal{H}^0_{\varepsilon}}.  \label{david-3}
\end{align}
First, we see that with an application of Lemma \ref{what-3} with \eqref{david-1} and \ref{david-2} there holds, for all $t\in[0,T]$,
\begin{align}
\sup_{\tau\ge1} \tau\mathbb{T}_\ep(\tau;\widehat{\Theta}^t) \le C(T)\left\| {\overline{\Delta}}_0 \right\|^2_{\mathcal{H}^0_{\varepsilon}},  \label{david-4}
\end{align}
and secondly, by applying the weak form of Lemma \ref{what-1} (see Remark \ref{r:what-trans}), we find that for all $t\in[0,T],$
\begin{align}
\ep\|{\rm{T}}_\ep\Phi^t\|^2_{\mathcal{M}^0_\ep} \le C(T)\left\| {\overline{\Delta}}_0 \right\|^2_{\mathcal{H}^0_{\varepsilon}}.  \label{david-5}
\end{align}
Together \eqref{david-4}-\eqref{david-5} establish \eqref{david-3}.
Therefore, inequality (\ref{difference-decomposition-K}) now follows with $R= \Xi(t^{\ast})=(\widehat{W}(t^*),\widehat{\Theta}^{t^*})$ and $\wp^{\ast}=M_2(t^{\ast})\geq 0$ (for a suitably updated function $M_2$). 
This finishes the proof of (C2).
\end{proof}

\begin{lemma}
\label{t:to-C3} Condition (C3) holds for each $\ep\in(0,1]$ and $\omega\in(0,1)$.
\end{lemma}

\begin{proof}
Let $\ep\in(0,1]$ and $\omega\in(0,1)$. 
Let $R>0$ and $\Upsilon_0=(U_0,\Phi_0)\in\mathcal{V}^1_{\varepsilon}$ where $\|\Upsilon_0\|_{\mathcal{V}^1_{\varepsilon}}\le R.$
Directly from (\ref{strong-bound}), there holds, 
\begin{equation*}
\left\| \mathcal{S}_{\varepsilon}(t)\Upsilon_0 \right\|_{\mathcal{V}^1_{\varepsilon}} \leq \widetilde{P}_1,
\end{equation*}
but where now the size of the initial data, $R$, depends on the size of $\mathcal{B}^{1}_{\varepsilon}$. 
Hence, on the compact interval $[t^{\ast },2t^{\ast }]$, the map $t\mapsto S(t)\Upsilon _{0}$ is Lipschitz continuous for each fixed $\Upsilon _{0}\in 
\mathcal{B}^{1}_{\varepsilon}$.
This means there is a constant $L=L(t^{\ast })>0$ such that
\begin{equation*}
\Vert \mathcal{S}_\ep(t_1)\Upsilon _{0} - \mathcal{S}_\ep(t_2)\Upsilon_{0}\Vert _{\mathcal{H}^0_{\varepsilon}} \leq L|t_1-t_2|.
\end{equation*}
Together with the continuous dependence estimate (\ref{cde}), (C3) follows.
\end{proof}

\begin{remark}
According to Proposition \ref{abstract1}, for each $\varepsilon\in(0,1]$, the semigroup $\mathcal{S}_{\varepsilon}(t):\mathcal{H}^0_{\varepsilon}\rightarrow \mathcal{H}^0_{\varepsilon}$ possesses an exponential attractor, $\mathfrak{M}_{\varepsilon}\subset \mathcal{B}^{1}_{\varepsilon}$, which attracts bounded subsets of $\mathcal{B}^{1}_{\varepsilon}$ exponentially fast (in the topology of $\mathcal{H}^0_{\varepsilon}$). 
Moreover, in light of the results in this section---which are {\em{uniform}} in the perturbation parameter $\ep$---we now simply accept the corresponding results for the simpler limit Problem {\textbf{P}}$_0$.
In this setting we use the notation for the compact absorbing set $\mathcal{B}^1_0$ and the exponential attractor $\mathfrak{M}_0$ admitted by the semigroup $\mathcal{S}_0(t):\mathcal{H}^0_0=\mathbb{X}^2\rightarrow\mathbb{X}^2$.
\end{remark}

\begin{remark}  \label{rem_att} 
In order to show that the attraction property (iii) in Theorem \ref{t:exponential-attractors} also holds---that is, in order to show that the basin of attraction of $\mathfrak{M}_\ep$ is all of $\mathcal{H}^0_\ep$---we appeal to the transitivity of the exponential attraction in Proposition \ref{t:exp-attr} and Theorem \ref{t:trans-out} below.
\end{remark}

%===========================================================
\subsection{Basin of attraction (and global attractors)}
%===========================================================

The main result in this section has two purposes: primary, per the above remark, it will help us show that the exponential attractors we seek attract every bounded subset in $\mathcal{H}^0_\ep$ (not just $\mathcal{B}^1_\ep$).
This property is sometimes not obvious because of the difficulties using spaces involving memory (we refer the reader to Section 1 of this article and to the rate of attraction of $\mathcal{B}^1_\ep$ as found in Lemma \ref{t:to-C1}).
However, we overcome this problem, partly, by proving a condition on the solution semigroup $\mathcal{S}_\ep$ that is also essential for the existence of global attractors (also called a universal attractors); we refer to the asymptotic compactness/smoothing of $\mathcal{S}_\ep$, which happens to occur in our case with an exponential rate. 
Together, the asymptotic compactness os $\mathcal{S}_\ep$ (Theorem \ref{t:trans-out} below) and the existence of an absorbing sets in $\mathcal{H}^0_\ep$ (Lemma \ref{weak-ball}) will guarantee the existence of a global attractor that is compact in $\mathcal{H}^0_\ep$ and bounded in $\mathcal{V}^1_\ep.$

\begin{theorem}  \label{t:trans-out}
For each $\ep\in[0,1],$ there is a positive constant $\varrho_1$ and a monotonically increasing function $Q(\cdot)$ in which for every nonempty bounded subset $B$ of $\mathcal{H}^0_\ep$ there holds, for all $t\ge0,$
\begin{equation*}
\mathrm{dist}_{\mathcal{H}^0_{\varepsilon }}(\mathcal{S}_{\varepsilon }(t)B,\mathcal{B}^1_{\varepsilon })\leq Q(\Vert B\Vert _{\mathcal{H}^0_{\varepsilon }})e^{-\varrho_1 t}.
\end{equation*}
\end{theorem}

\begin{proof}
Because of the smoothing properties of the associated with the Wentzell parabolic Problem {\textbf{P}}$_0$ (cf. \cite{Gal12-1}), we limit ourselves to the case when $\ep\in(0,1].$

Let $\ep\in(0,1]$ and $B$ be a nonempty bounded subset of $\mathcal{H}^0_\ep.$
By recalling Lemma \ref{weak-ball}, we already know that there is a bounded absorbing set that is exponentially attracting in $\mathcal{H}^0_\ep$, i.e., for all $t\ge0$ there holds
\begin{equation*}
\mathrm{dist}_{\mathcal{H}^0_{\varepsilon }}(\mathcal{S}_{\varepsilon }(t)B,\mathcal{B}^0_{\varepsilon }) \leq Q(\Vert B\Vert _{\mathcal{H}^0_{\varepsilon }})e^{-\nu_0 t},
\end{equation*}
so owing once again to the transitivity of exponential attraction (cf. Proposition \ref{t:exp-attr} below) it suffices to show that, for all $t\ge0,$
\begin{equation}
\mathrm{dist}_{\mathcal{H}^0_{\varepsilon }}(\mathcal{S}_{\varepsilon }(t)\mathcal{B}^0_\ep,\mathcal{B}^1_{\varepsilon })\leq Q(P_0)e^{-\varrho_0 t},  \label{smooth-1}
\end{equation}
for some positive constant $\varrho_0$ and for some positive monotonically increasing function $Q(\cdot)$, each independent of $\ep.$
(Recall from \eqref{ball-0} that $\sqrt{P_0+1}$ is the radius of $\mathcal{B}^0_\ep$.) 

To prove \eqref{smooth-1}, the idea is to show that for each $\ep\in(0,1]$ and for each $\Upsilon_0\in\mathcal{H}^0_\ep$ we can decompose the semigroup
\[
\mathcal{S}_\ep(t)\Upsilon_0 = \mathcal{Z}_\ep(t)\Upsilon_0 + \mathcal{K}_\ep(t)\Upsilon_0
\]
where the operators $\mathcal{Z}_\ep$ are uniformly (exponentially) decaying to zero and $\mathcal{K}_\ep$ are uniformly compact (bounded in $\mathcal{V}^1_\ep$) for large $t$. 
This is done in the following lemmas.
\end{proof}

The following decomposition and subsequently more general lemmas, as we will allow the datum to belong to {\em{any}} bounded subset of the phase space $\mathcal{H}^0_\ep$, can be seen to follow \cite[Theorem 6.10--Lemma 6.12]{CPS06} with obvious changes to account for the dynamic boundary conditions with memory.
Hence, we will limit the proofs to sketches of the most important details. 

First, choose a constant $M_F>0$, based on \eqref{assm-3}, \eqref{assm-4}, and \eqref{func}, so that the map defined by, for all $s\in\mathbb{R},$
\begin{equation*}
F_0(s):=F(s)+M_Fs,
\end{equation*}
satisfies, for every $s\in\mathbb{R},$ 
\begin{equation}
F'_0(s)\ge \binom{0}{0}.  \label{up}
\end{equation}
Next, let $\Upsilon_{0}=(U_{0},\Phi_{0})\in \mathcal{H}^0_{\varepsilon }$.
Then rewrite Problem {\textbf{P}}$_\ep$ into the system of equations in $(V,\Psi)$ and $(W,\Theta)$, where $(V,\Psi)+(W,\Theta)=(U,\Phi)$,
\begin{equation}  \label{decomp-v}
\left\{ \begin{array}{l}
\partial_t V(t) + \omega {\mathrm{A^{0,\beta}_{W}}}V(t) + \displaystyle\int_0^\infty \mu_{\varepsilon}(s) {\mathrm{A^{\alpha,\beta}_{W}}} \Psi^t(s) \diff s + F_0(U(t)) - F_0(W(t)) = 0, \\ 
\partial_t\Psi^t(s) = {\rm{T}_{\varepsilon}} \Psi^t(s) + V(t), \\
(V(0),\Psi^0)=\Upsilon_0,
\end{array} \right.
\end{equation}
and
\begin{equation}  \label{decomp-w}
\left\{ \begin{array}{l}
\partial_t W(t) + \omega {\mathrm{A^{0,\beta}_{W}}}W(t) + \displaystyle\int_0^\infty \mu_{\varepsilon}(s) {\mathrm{A^{\alpha,\beta}_{W}}} \Theta^t(s) \diff s + F_0(W(t)) - M_F U(t) = 0, \\ 
\partial_t \Theta^t(s) = {\rm{T}_{\varepsilon}} \Theta^t(s) + W(t), \\
(W(0),\Theta^0)={\mathbf{0}}.
\end{array} \right.
\end{equation}

In view of Lemmas \ref{t:decay} and \ref{t:compact} below,
we define the one-parameter family of maps, $\mathcal{K}_{\varepsilon }(t):\mathcal{H}^0_{\varepsilon }\rightarrow \mathcal{H}^0_{\varepsilon }$, by 
\begin{equation*}
\mathcal{K}_{\varepsilon }(t)\Upsilon_0:=( W(t),\Theta^t),
\end{equation*}
where $(W,\Theta)$ is a solution of (\ref{decomp-w}). 
With such $(W,\Theta)$, we may define a second function $(V,\Psi)$ as the solution of (\ref{decomp-v}). 
Through the dependence of $(V,\Psi)$ on $(W,\Theta)$ and $(U(0),\Phi^0)=\Upsilon_0$, the solution of (\ref{decomp-v}) defines a
one-parameter family of maps, $\mathcal{Z}_{\varepsilon}(t):\mathcal{H}^0_{\varepsilon}\rightarrow \mathcal{H}^0_{\varepsilon }$, defined by 
\begin{equation*}
\mathcal{Z}_{\varepsilon}(t)\Upsilon_0:=(V(t),\Psi^t).
\end{equation*}
Notice that if $(V,\Psi)$ and $(W,\Theta)$ are solutions to (\ref{decomp-v}) and (\ref{decomp-w}), respectively, then the function $(U(t),\Phi^t):=(V(t),\Psi^t)+(W(t),\Theta^t)$ is a solution to Problem {\textbf{P}$_\ep$.

The next result shows that the operators $\mathcal{Z}_{\varepsilon }$ are uniformly decaying to zero in $\mathcal{H}_{\varepsilon }$.

\begin{lemma}  \label{t:decay} 
For each $\varepsilon \in (0,1]$ and $\Upsilon_{0}=(U_{0},\Phi_{0})\in \mathcal{H}^0_{\varepsilon }$, there exists a unique global weak solution $(V,\Psi)\in C([0,\infty);\mathcal{H}^0_{\varepsilon })$ to problem (\ref{decomp-v}).
Moreover, given $R>0$, then for all $\Upsilon_0\in\mathcal{H}^0_\varepsilon$ with $\|\Upsilon_{0}\|_{\mathcal{H}^0_{\varepsilon}}\leq R$ for all $\varepsilon \in (0,1]$, there exists $\nu_0'>0 $, independent of $\varepsilon $, such that, for all $t\geq 0$, 
\begin{equation}
\|\mathcal{Z}_{\varepsilon}(t)\Upsilon_{0}\|_{\mathcal{H}^0_{\varepsilon}}\leq
Q(R)e^{-\nu_0' t}.  \label{eq:uniform-decay}
\end{equation}
\end{lemma}

\begin{proof}
The existence of a global weak solution to (\ref{decomp-v}) follows as the proof of \cite[Theorem \ref{t:weak-solutions}]{Gal&Shomberg15}.
It remains to show that (\ref{eq:uniform-decay}) holds.

The proof is very similar to the proof of Lemma \ref{weak-ball} save that the assumptions \eqref{assm-3}-\eqref{assm-4} become crucial.
Indeed, the constant $C$ on the right-hand side of \eqref{fus-10} vanishes because nonlinear terms now satisfy the bound
\[
\langle F_0(U)-F_0(W),V \rangle_{\mathbb{X}^2} \ge 0
\]
as here $V=U-W$ and \eqref{up} holds.
\end{proof}

The following lemma establishes the uniform compactness of the operators $\mathcal{K}_{\varepsilon }$.

\begin{lemma}  \label{t:compact} 
For each $\varepsilon \in (0,1]$ and $\Upsilon_{0}=(U_{0},\Phi_{0})\in \mathcal{H}^0_{\varepsilon}$, there exists a unique global weak solution $(W,\Theta)\in C([0,\infty);\mathcal{H}^0_{\varepsilon})$ to problem (\ref{decomp-w}).
Moreover, given $R>0$, then for all $\Upsilon_{0}\in \mathcal{H}^0_{\varepsilon }$ with $\|\Upsilon_0\|_{\mathcal{H}^0_{\varepsilon}} \leq R$ for all $\varepsilon \in (0,1]$, there holds for all $t\geq 0$, 
\begin{equation*}
\|\mathcal{K}_{\varepsilon}(t)\Upsilon_{0}\|_{\mathcal{V}^1_{\varepsilon }}\leq Q(R),
\end{equation*}
Furthermore, the operators $K_{\varepsilon }$ are uniformly compact in $\mathcal{H}^0_{\varepsilon }$.
\end{lemma}

\begin{proof}
Again, in light of \cite[Theorem \ref{t:weak-solutions}]{Gal&Shomberg15}, it remains to show that the operators $\mathcal{K}_{\varepsilon }$ are uniformly compact in $\mathcal{H}^0_{\varepsilon }$.

This time we appeal to Lemma \ref{t:to-C1} whereby only trivial changes are required in the proof in order to show Lemma \ref{t:compact} holds.
\end{proof}

\begin{remark}
These results---with datum contained to the absorbing set $\mathcal{B}^0_\ep$---complete the proof of Theorem \ref{t:trans-out}.
Consequently, the existence of a (finite dimensional) global attractor $\mathcal{A}_\ep$, $\ep\in(0,1]$, for $\mathcal{S}_\ep$ follows.
\end{remark}

\begin{theorem} \label{t:global-attractors}
For each $\varepsilon \in (0,1]$, the semigroup $\mathcal{S}_\ep$ admits a unique global attractor
\begin{equation*}
\mathcal{A}_{\varepsilon } = \omega (\mathcal{B}^0_{\varepsilon}) := \bigcap_{s\geq 0}{\overline{\bigcup_{t\geq s}\mathcal{S}_{\varepsilon }(t)\mathcal{B}^0_{\varepsilon}}}^{\mathcal{H}^0_{\varepsilon }}
\end{equation*}
in $\mathcal{H}^0_{\varepsilon }$. 
Moreover, the following hold:
\begin{description}

\item[1] For each $t\geq 0$, $\mathcal{S}_{\varepsilon }(t)\mathcal{A}_{\varepsilon } = \mathcal{A}_{\varepsilon }$, and

\item[2] For every nonempty bounded subset $B$ of $\mathcal{H}^0_{\varepsilon }$,
\begin{equation}
\lim_{t\rightarrow \infty }\mathrm{dist}_{\mathcal{H}^0_{\varepsilon
}}(\mathcal{S}_{\varepsilon }(t)B,\mathcal{A}_{\varepsilon })=0.  \label{gl_attraction}
\end{equation}

\item[3] The global attractor $\mathcal{A}_{\varepsilon }$ is bounded in $
\mathcal{V}^1_{\varepsilon }$ (hence, compact in $\mathcal{H}^0_\ep$) and trajectories on $\mathcal{A}_{\varepsilon }$ are strong solutions (in the sense of Definitions \ref{d:strong-solution}).

\item[4] The fractal dimension is bounded, uniformly in $\ep$, i.e.,
\[
\dim_{\mathrm{F}}(\mathcal{A}_\ep,\mathcal{H}^0_\ep) \le \dim_{\mathrm{F}}(\mathfrak{M}_\ep,\mathcal{H}^0_\ep) \le C < \infty,
\]
for some constant $C>0$ independent of $\ep.$
\end{description}
\end{theorem}

\begin{proof}
The existence and boundedness of the global attractor for Problem {\textbf{P}}$_0$ can be found in \cite[Theorem 2.3]{Gal12-2} and the references therein.
Thus, it suffices to show the result for the perturbation Problem {\textbf{P}}$_\ep$, with $\ep\in(0,1].$
By referring to the standard literature (cf. e.g. \cite{Babin&Vishik92,Temam88}) and Lemma \ref{weak-ball}, Lemma \ref{t:to-C1}, and Theorem \ref{t:trans-out}, the proof is complete.
\end{proof}

%===========================================================
\subsection{Robustness and H\"{o}lder continuity of the exponential attractors}
%===========================================================

What remains in this section is to show that the family of exponential attractors is robust, or H\"{o}lder continuous with respect to the perturbation parameter $\ep$. 
As a preliminary step, we follow, for example \cite[see p. 177]{CPS05} among others, and define the so-called {\em{canonical extension}} map, $\mathcal{E}:\mathbb{X}^2 \rightarrow \mathfrak{M}_\varepsilon$, by
\begin{equation}
\label{canonical-extension-map-2}
\mathcal{E}(U) = 0.
\end{equation}
With this, define the {\em{lift}} mapping, $\mathcal{L}\mathbb{X}^2 \rightarrow \mathcal{H}^0_\varepsilon$, by
\begin{equation}
\label{lift}
\mathcal{L}(U) = (U,\mathcal{E}(U)) = (U,0).
\end{equation}

\begin{theorem}
\label{cont_att}Let the assumptions of Theorem \ref{t:exponential-attractors} be satisfied. 
For each $\varepsilon$, the semigroup of solution operators, $\mathcal{S}_\varepsilon(t)$ admits an exponential attractor $\mathfrak{M}_\varepsilon$ in which the family of compact sets $(\mathbb{M}_{\varepsilon }) _{\varepsilon \in \left[ 0,1\right] }$ defined by 
\begin{equation}\label{global-attractors-family}
\mathbb{M}_\varepsilon := \left\{ \begin{array}{ll}
\mathcal{L}\mathfrak{M}_0 & \text{for}\quad\varepsilon=0 \\ 
\mathfrak{M}_\varepsilon & \text{for}\quad\varepsilon\in(0,1]
\end{array} \right.
\end{equation}
is H\"{o}lder continuous for every $\varepsilon \in \left[ 0,1\right] $, i.e., there exist constants $\Lambda >0,$ $\tau \in (0,1/2]$ independent of $\varepsilon $, such that, for every $0\leq \varepsilon _{2}<\varepsilon _{1}\leq 1$, the symmetric Hausdorff distance satisfies
\begin{equation}  \label{symm-diff}
\mathrm{dist}_{\mathcal{H}^0_{\varepsilon _{1}}}^{\mathrm{sym}}(\mathfrak{M}_{\varepsilon _{1}},\mathfrak{M}_{\varepsilon _{2}}) \leq \Lambda (\varepsilon _{1}-\varepsilon _{2})^{\tau }.
\end{equation}
\end{theorem}

\begin{remark}
The {\em{symmetric Hausdorff distance}} between two subsets $A,B$ of a Banach space $X$ is defined as 
\begin{align}
\mathrm{dist}_{X}^{\mathrm{sym}}(A,B) := \max\left\{ \mathrm{dist}_{X}(A,B),\mathrm{dist}_{X}(B,A) \right\}.  \notag
\end{align}
More precisely, the condition given in (\ref{symm-diff}) implies the family of attractors is both upper- and lower-semicontinuous (thus, continuous) at each value of the perturbation parameter $\varepsilon\in[0,1)$. 
\end{remark}

In order to prove Theorem \ref{cont_att} we will develop the main assumptions of the abstract results found in the seminal works \cite{GGMP05,GMPZ10}.
As in Proposition \ref{abstract1} above, the assumptions suited specifically for our needs appear in \cite[(H2) and (H3) of Theorem A.2]{CPS06}.

As above, the number $L>0$ shown below is used to denote the (local) Lipschitz constant of the mapping $F:\mathbb{V}^1\rightarrow\mathbb{X}^2$.

\begin{proposition}  \label{abstract2}
Let the assumptions of Proposition \ref{abstract1} be
satisfied. In addition, assume the following:

\begin{enumerate}
\item[(C4)] The canonical extension map $\mathcal{E}_{\mid\mathcal{B}_{0}^{1}}:\mathbb{X}^{2} \rightarrow \mathcal{H}^0_\varepsilon$ given by (\ref{canonical-extension-map-2}) is Lipschitz continuous.

\item[(C5)] There is a constant $\Lambda_1=\Lambda_1(L,\Omega ,t^{\ast })>0$ such that, for all $t\in \lbrack t^{\ast },2t^{\ast }]$ and for all $\Upsilon_{0}=(U_{0},\Phi_0)\in\mathcal{B} _{\varepsilon }^{1}$,
\begin{equation}
\left\Vert \mathcal{S}_{\varepsilon }(t)\Upsilon_0 - \mathcal{LS}_{0}(t)\mathbb{P} \Upsilon_{0} \right\Vert _{\mathcal{H}^0_{\varepsilon }} \leq \Lambda_1\sqrt{\varepsilon }.
\label{robust-part}
\end{equation}
Here, $\mathbb{P} :\mathcal{H}^0_{\varepsilon }\rightarrow \mathcal{H}^0_{0}$ denotes the projection defined by, for all $\Upsilon = (U,\Phi)\in \mathcal{H}^0_{\varepsilon }$,
\begin{equation*}
\mathbb{P} \Upsilon = U.
\end{equation*}

\item[(C6)] There is a constant $\Lambda_{2}=\Lambda_{2}(L,\Omega ,t^{\ast })>0$ such that, for all $t\in \lbrack t^{\ast },2t^{\ast }]$, $\Upsilon_0 = (U_{0},\Phi_{0})\in \mathcal{B}_{\varepsilon _{1}}^{1}\subset \mathcal{H}^1_{\varepsilon_1}$
and, for all $0<\varepsilon_2<\varepsilon_1\leq 1$, 
\begin{equation}
\left\Vert \mathcal{S}_{\varepsilon _{1}}(t)\Upsilon_{0} - \mathcal{S}_{\varepsilon _{2}}(t)\Upsilon_{0} \right\Vert _{\mathcal{H}^0_{\varepsilon_1}} \leq
\Lambda_{2}(\varepsilon_1-\varepsilon_2)^{1/2}.  \label{Holder-part}
\end{equation}
\end{enumerate}

Then, the family of exponential attractors $(\mathbb{M}_{\varepsilon})_{\varepsilon \in \left[ 0,1\right] }$ is H\"{o}lder continuous for every $\varepsilon \in \left[ 0,1\right] $ in the sense of Theorem \ref{cont_att}.
\end{proposition}

\begin{remark}
The condition (C6) below does not appear in \cite{CPS06}, but rather we now borrow \cite[(H7) of Theorem 4.4]{GMPZ10}, cf. also \cite[(P4) of Theorem 2.1]{MPZ07}.
\end{remark}

\begin{lemma}
Condition (C4) holds.
\end{lemma}

\begin{proof}
Based on the definition of $\mathcal{E}$ given in (\ref{canonical-extension-map-2}), the result is vacuously true.
\end{proof}

The following lemma proves condition (C5) of Proposition \ref{abstract2}. 
It shows that the difference between the semigroups $\mathcal{S}_{\varepsilon}(t)$ and the lifted limit semigroup $\mathcal{LS}_{0}(t)$ in $\mathcal{H}^0_{\varepsilon }$, on finite time intervals, is of order $\varepsilon^{1/2}$.

\begin{lemma}
Let $T>0$. For all $\varepsilon\in(0,1]$, $\omega\in(0,1)$ and $\Upsilon_0=(U_0,\Phi_0)\in\mathcal{H}^0_\varepsilon$ such that $\|\Upsilon_0\|_{\mathcal{H}^0_\varepsilon} \leq R$ for all $\varepsilon\in(0,1]$, there exists a positive constant $C(T)$, independent of $\varepsilon$, but depending on $\omega$ and $T$, in which, for all $t\geq [0,T]$, 
\begin{equation}
\label{C5-key-1}
\left\| \mathcal{S}_\varepsilon(t)\Upsilon_0 - \mathcal{LS}_0(t) \mathbb{P} \Upsilon_0 \right\|_{\mathcal{H}^0_\varepsilon} \leq C(T)\varepsilon^{1/2}.
\end{equation}
\end{lemma}

\begin{proof}
Let $\widehat{\Upsilon}(t)=(\widehat{U}(t),\widehat{\Phi}^t)$ denote the solution of Problem P$_\varepsilon$ corresponding to the initial data $\Upsilon_0=(U_0,\Phi_0)\in \mathcal{B}_{\varepsilon }^{1}$ and let $U(t)$ denote the solution of Problem P$_0$ corresponding to the initial data $\mathbb{P} \Upsilon_{0} = U_{0}\in \mathcal{B}_{0}^{1}$. 
With the solution $U(t)$, define the function $\Phi^t$ by the solution to the Cauchy problem,
\begin{equation}
\label{C5-key-2}
\left\{ \begin{array}{l}
\partial_t\Phi^t = \mathrm{T}_\varepsilon\Phi^t + U(t) \\
\Phi^0 = \mathbb{Q}\Upsilon_0=\Phi_0 \in \mathcal{M}^0_\varepsilon.
\end{array} \right.
\end{equation}
With the (unique) solution to (\ref{C5-key-2}) (cf. Corollary \ref{t:memory-regularity-1}), define $\Upsilon(t) := (U(t),\Phi^t)$ for all $t\geq 0$.
Let 
\begin{equation*}\begin{aligned}
\widehat{\Delta}(t) = (Z(t),\Theta^t) & := \widehat{\Upsilon}(t) - \Upsilon(t) \\ 
& = (\widehat{U}(t),\widehat{\Phi}^t) - (U(t),\Phi^t) \\
& = (\widehat{U}(t) - U(t), \widehat{\Phi}^t - \Phi^t);
\end{aligned}\end{equation*}
hence, $\widehat{\Delta}(t)=(Z(t),\Theta^t)$ satisfies the system
\begin{equation}
\label{C5-key-3}
\left\{ \begin{array}{l}
\partial_t Z(t) + \omega \mathrm{A_{W}^{0,\beta}}Z(t) 
+ \displaystyle\int_0^\infty \mu_\varepsilon(s) \mathrm{A_{W}^{\alpha,\beta}} \Theta^t(s) \diff s + F(\widehat{U}(t)) - F(U(t)) = -\int_0^\infty \mu_\varepsilon(s) \mathrm{A_{W}^{\alpha,\beta}} \Phi^t(s) \diff s, \\ 
\partial_t\Theta^t(s) = {\rm{T_\varepsilon}} \Theta^t(s) + Z(t), \\
\left( Z(0),\Theta^0 \right) = {\bf{0}}.
\end{array} \right.
\end{equation}
Multiply (\ref{C5-key-3})$_1$ by $Z$ in $\mathbb{X}^2$ and (\ref{C5-key-3})$_2$ by $\mathrm{A_{W}^{\alpha,\beta}}\Theta^t$ in $L^2_{\mu_\varepsilon}(\mathbb{R}_+;\mathbb{X}^2)$, summing the resulting identities and estimating as in the above arguments, it is not hard to see that there holds, for almost all $t\geq 0$,
\begin{equation}\begin{aligned}
\label{C5-key-4}
& \frac{1}{2}\frac{\diff}{\diff t} \left\{ \left\| Z \right\|^2_{\mathbb{X}^2} + \left\| \Theta^t \right\|^2_{\mathcal{M}^1_\varepsilon} \right\} + \omega\left\| Z \right\|^2_{\mathbb{V}^1} + \frac{\delta}{2\varepsilon}\left\| \Theta^t \right\|^2_{\mathcal{M}^1_\varepsilon} \\ 
& \leq -\left\langle F(\widehat{U}) - F(U),Z \right\rangle_{\mathbb{X}^2} + \omega\alpha\|z\|^2_{L^2(\Omega)} - \int_0^\infty \mu_\varepsilon(s) \left\langle \mathrm{A_{W}^{\alpha,\beta}}\Phi^t(s),Z \right\rangle_{\mathbb{X}^2} \diff s.
\end{aligned}\end{equation}
Recall, with (\ref{assm-3}) we obtain, 
\begin{equation}
\label{C5-key-5}
-\left\langle F(\widehat{U})-F(U),Z \right\rangle_{\mathbb{X}^2} \leq M_F \left\| Z \right\|^2_{\mathbb{X}^2}.
\end{equation}
For the remaining term on the right-hand side, we apply the definition of the norm and Young's inequality to find, 
\begin{equation*}\begin{aligned}
- \int_0^\infty \mu_\varepsilon(s) \left\langle \mathrm{A_{W}^{\alpha,\beta}} \Phi^t(s),Z \right\rangle_{\mathbb{X}^2} \diff s & = -\int_0^\infty \mu_\varepsilon(s) \left\langle \Phi^t(s),Z \right\rangle_{\mathbb{V}^1} \diff s \\
& \leq \|Z\|_{\mathbb{V}^1} \left\| \Phi^t \right\|_{\mathcal{M}^1_\varepsilon} \\
& \leq \omega\left\| Z \right\|^2_{\mathbb{V}^1} + \frac{1}{4\omega} \left\| \Phi^t \right\|^2_{\mathcal{M}^1_\varepsilon}.
\end{aligned}\end{equation*}
Recall that, by (\ref{Phi-decay-1}), there holds for all $t\geq 0$,
\begin{equation}
\label{C5-key-14}
\left\| \Phi^t \right\|^2_{\mathcal{M}^1_\varepsilon} \leq \left\| \Phi_0 \right\|^2_{\mathcal{M}^1_\varepsilon} e^{-\delta t/2\varepsilon} + C\varepsilon,
\end{equation}
where $C>0$ depends on the bound $\widetilde{P}_0$, but is uniform in $\varepsilon$ and $t$. 
Collecting (\ref{C5-key-4})-(\ref{C5-key-14}) yields,
\begin{equation}\begin{aligned}
\label{C5-key-8}
\frac{\diff}{\diff t} & \left\{ \left\| Z \right\|^2_{\mathbb{X}^2} + \left\| \Theta^t \right\|^2_{\mathcal{M}^1_\varepsilon} \right\} + \delta\left\| \Theta^t \right\|^2_{\mathcal{M}^1_\varepsilon} \\ 
& \leq 2\left( \omega\alpha + M_F \right) \left\| Z \right\|^2_{\mathbb{X}^2} + C \varepsilon \left\| Z \right\|^2_{\mathbb{X}^2} + C_\omega \left( \left\| \Phi_0 \right\|^2_{\mathcal{M}^1_\varepsilon} e^{-\delta t/2\varepsilon} + \varepsilon \right).
\end{aligned}\end{equation}
Integrating (\ref{C5-key-8}) with respect to $t$ on the interval $[0,T]$, for $T>0$, and then applying the initial conditions (\ref{C5-key-4})$_3$, as well as the uniform bound (\ref{weak-bound}), we have, 
\begin{equation}
\label{C5-key-9}
\left\| Z(t) \right\|^2_{\mathbb{X}^2} + \left\| \Theta^t \right\|^2_{\mathcal{M}^1_\varepsilon} \leq \int_0^t C \left\|Z(\tau) \right\|^2_{\mathbb{X}^2} \diff\tau + C(T)\varepsilon.
\end{equation}

Next we seek an appropriate bound on the term with $Z$. It follows from (\ref{C5-key-9}) and Gronwall's inequality that there holds, for all $t\geq 0$ and for all $\varepsilon\in(0,1]$,
\begin{equation}\label{C5-key-17}
\|Z(t)\|^2_{\mathbb{X}^2} \leq C(T)\varepsilon,
\end{equation}
where $C>0$ depends on $\omega$, $\delta$, and of course $T$, but not $\varepsilon$.

Returning to (\ref{C5-key-9}), we now see that there holds, for all $t\in[\sqrt{\varepsilon},T]$ and for all $\varepsilon\in(0,1]$,
\begin{equation}
\label{C5-key-12}
\left\| \left( Z(t), \Theta^t \right) \right\|^2_{\mathcal{H}^0_\varepsilon} \leq C(T)\varepsilon.
\end{equation}
Therefore (\ref{C5-key-1}) follows. This finishes the proof.
\end{proof}

We will establish the H\"{o}lder continuity with the following lemma.
With regard to \cite{GMPZ10}, in particular, hypothesis (H7) of Theorem 4.4 there, we {\em{do not}} perform an $\varepsilon$-scaling of the memory variable.

\begin{lemma}
Condition (C6) holds.
\end{lemma}

\begin{proof}
Assume $0<\varepsilon _{2}<\varepsilon _{1}\leq 1$. 
Let $\Upsilon_0 = (U_0,\Phi^0) \in \mathcal{B}_{1}^{1}$. 
Let $\widetilde{\Upsilon}(t)=(\widetilde{U}(t),\widetilde{\Phi}^t)$ denote the solution of Problem P$_{\varepsilon_1}$ corresponding to the initial datum $\Upsilon_{0}$ and let $\widetilde{\Xi}(t)=(\widetilde{V}(t),\widetilde{\Psi}^t)$ denote the solution Problem P$_{\varepsilon_2}$ corresponding to the same initial datum $\Upsilon_{0}$. Let 
\begin{equation*}\begin{aligned}
\widetilde{\Delta}(t) = \left( \widetilde{Z}(t),\widetilde{\Theta}^t \right) & := \widetilde{\Upsilon}(t) - \widetilde{\Xi}(t) \\
& = \left(\widetilde{U}(t),\widetilde{\Phi}^t\right) - \left(\widetilde{V}(t),\widetilde{\Psi}^t\right) \\
& = \left( \widetilde{U}(t) - \widetilde{V}(t),\widetilde{\Phi}^t - \widetilde{\Psi}^t \right).
\end{aligned}\end{equation*}
\begin{equation}
\label{tilde-z-difference}
\left\{ 
\begin{array}{l}
\partial_t\widetilde{Z} + \omega \mathrm{A_{W}^{0,\beta}} \widetilde{Z} + \displaystyle\int_0^\infty \mu_{\varepsilon_1}(s) \mathrm{A_{W}^{\alpha,\beta}} \widetilde{\Theta}^t(s)\diff s + F\left(\widetilde{U}\right) - F\left(\widetilde{V}\right) = \int_0^\infty \left( \mu_{\varepsilon_2}(s) - \mu_{\varepsilon_1}(s) \right) \mathrm{A_{W}^{\alpha,\beta}} \widetilde{\Psi}^t(s)\diff s \\ 
\partial_t\widetilde{\Theta}^t(s) = {\rm{T}}_{\varepsilon_1}\widetilde{\Theta}^t(s) + \widetilde{Z}(t) \\
\widetilde{Z}(0) = {\bf{0}}, \quad\widetilde{\Theta}^0 = {\bf{0}}.
\end{array}
\right.
\end{equation}
Observe, by the definition of ${\rm{T}}_\varepsilon$, $\left( {\rm{T_{\varepsilon_1}}} - {\rm{T}}_{\varepsilon_2} \right)\widetilde{\Psi}^t(s)=0$.
We proceed in the usual fashion by multiplying (\ref{tilde-z-difference})$_1$ by $\widetilde{Z}$ in $\mathbb{X}^2$, and multiplying equation (\ref{tilde-z-difference})$_2$ by $\mathrm{A_{W}^{\alpha,\beta}} \widetilde{\Theta}^t$ in $L^2_{\mu_{\varepsilon_1}}(\mathbb{R}_+;\mathbb{X}^2)$,
summing the results, we arrive at the identity, 
\begin{equation}\begin{aligned}
\label{C6-1}
\frac{1}{2}\frac{\diff}{\diff t} & \left\{ \left\| \widetilde{Z} \right\|^2_{\mathbb{X}^2} + \left\| \widetilde{\Theta}^t \right\|^2_{\mathcal{M}^0_{\varepsilon_1}} \right\} + \\
& + \omega\left\|\widetilde{Z} \right\|^2_{\mathbb{V}^1} - \int_0^\infty \mu_{\varepsilon_1}(s)\left\langle {\rm{T_{\varepsilon_1}}}\widetilde{\Theta}^t(s), \mathrm{A_{W}^{\alpha,\beta}} \widetilde{\Theta}^t(s) \right\rangle_{\mathbb{X}^2} \diff s \\
& = \int_0^\infty \left( \mu_{\varepsilon_2}(s) - \mu_{\varepsilon_1}(s) \right) \left\langle \widetilde{\Psi}^t(s), \widetilde{Z}(t) \right\rangle_{\mathbb{X}^2} \diff s + \\
& - \left\langle F\left(\widetilde{U}\right) - F\left(\widetilde{V}\right), \widetilde{Z} \right\rangle_{\mathbb{X}^2} + \omega\alpha \|\tilde{z}\|^2_{L^2(\Omega)}.
\end{aligned}\end{equation}
We estimate from here along the usual lines to obtain, for almost all $t\geq0$, 
\begin{equation}\begin{aligned}
\label{C6-2}
- \int_0^\infty \mu_{\varepsilon_1}(s)\left\langle {\rm{T_{\varepsilon_1}}}\widetilde{\Theta}^t(s), \mathrm{A_{W}^{\alpha,\beta}} \widetilde{\Theta}^t(s) \right\rangle_{\mathbb{X}^2} \diff s \le \frac{\delta}{2\varepsilon_1}\left\| \widetilde{\Theta}^t \right\|^2_{\mathcal{M}^1_{\varepsilon_1}}. 
\end{aligned}\end{equation}
We know there is a constant $M_F>0$ in which,
\begin{equation}
\label{C6-4}
- \left\langle F\left(\widetilde{U}\right) - F\left(\widetilde{V}\right), \widetilde{Z} \right\rangle_{\mathbb{X}^2} \leq M_2 \left\|\widetilde{Z} \right\|^2_{\mathbb{X}^2},
\end{equation}
and finally, with the fact that $\widetilde{\Psi}^t$ is uniformly bounded in $\mathcal{B}^1_{\varepsilon_1}$,
\begin{equation}\begin{aligned}
\label{C6-5}
\int_0^\infty \left( \mu_{\varepsilon_2}(s) - \mu_{\varepsilon_1}(s) \right) \left\langle \widetilde{\Psi}^t(s), \widetilde{Z}(t) \right\rangle_{\mathbb{X}^2} \diff s & = \frac{\varepsilon_1 - \varepsilon_2}{\varepsilon_1\varepsilon_2}\int_0^\infty \mu_{\varepsilon}(s) \left\langle \widetilde{\Psi}^t(s),\widetilde{Z}(t) \right\rangle_{\mathbb{X}^2} \diff s \\ 
& \leq C\frac{\varepsilon_1 - \varepsilon_2}{\varepsilon_2} \left\| \widetilde{Z} \right\|_{\mathbb{X}^2} \left\| \widetilde{\Psi}^t \right\|_{\mathcal{M}^1_{\varepsilon_1}} \\
& \leq \frac{\varepsilon_1 - \varepsilon_2}{\varepsilon_2}Q(R_1) + \frac{1}{2}\left\| \widetilde{Z} \right\|^2_{\mathbb{X}^2},
\end{aligned}\end{equation}
where $R_1>0$ is the radius of the absorbing set $\mathcal{B}^1_{\varepsilon_1}$.
After applying (\ref{C6-2})-(\ref{C6-5}), we obtain the differential inequality,
\begin{equation}\begin{aligned}
\label{C6-6}
\frac{\diff}{\diff t} & \left\{ \left\| \widetilde{Z} \right\|^2_{\mathbb{X}^2} + \left\| \widetilde{\Theta}^t \right\|^2_{\mathcal{M}^1_{\varepsilon_1}} \right\} \leq 2\left( M_2+\omega+1 \right) \left\|\widetilde{Z} \right\|^2_{\mathbb{X}^2} + C \left\| \widetilde{\Theta}^t \right\|^2_{\mathcal{M}^1_{\varepsilon_1}} + \frac{\varepsilon_1 - \varepsilon_2}{\varepsilon_2}Q(R_1),
\end{aligned}\end{equation}
where $M_3:=\max\{2\left( M_2+\omega+1 \right),C\}>0$.
We now integrate (\ref{C6-6}) with respect to $t$ over $[0,T]$ which in turn yields the Gronwall-type estimate, for all $t\in[0,T]$
\[
\left\| \left(\widetilde{Z}(t),\widetilde{\Theta}^t \right) \right\|_{\mathcal{H}^0_{\varepsilon_1}} \leq \sqrt{\frac{\varepsilon_1 - \varepsilon_2}{\varepsilon_2}}\frac{Q(R_1) \left( e^{M_3 T} - 1 \right)}{M_3}.
\]
Therefore, (\ref{Holder-part}) follows.
\end{proof}

\begin{remark}
In conclusion, by Theorem \ref{cont_att} the semigroup $\mathcal{S}_\varepsilon$ generated by the solutions of Problem P$_\varepsilon$ admits a robust family of exponential attractors $(\mathbb{M}_\varepsilon)_{\varepsilon\in[0,1]}$ in $\mathcal{H}^0_\varepsilon$, H\"{o}lder continuous at each $\varepsilon\in[0,1]$. 
\end{remark}

%===========================================================
\appendix
\section{}
%===========================================================

For the reader's convenience we report some important results that are needed in the article.

The following lemma is from \cite[Lemma 2.2]{Gal&Grasselli08}. It is in the spirit of the $H^s$-elliptic regularity estimate that can be found in \cite[Theorem II.5.1]{Lions&Magenes72}. 

\begin{lemma}  \label{t:appendix-lemma-3}
Consider the linear boundary value problem,
\begin{equation}
\label{appendix-BVP}
\left\{\begin{array}{rl}
-\Delta u+\alpha u & = \psi_1 \quad\text{in}\quad\Omega, \\ 
-\Delta_{\Gamma}u + \partial_{\mathbf{n}} u + \beta u & = \psi_2 \quad\text{on}\quad\Gamma.
\end{array}\right.
\end{equation}
If $(\psi_1,\psi_2)^{\mathrm{tr}}\in H^s(\Omega)\times H^s(\Gamma)$, for $s\geq 0$ and $s+\frac{1}{2} \not\in\mathbb{N}$, then the following estimate holds for some constant $C>0$,
\begin{equation}\label{H2-regularity-estimate}
\|u\|_{H^{s+2}(\Omega)} + \|u\|_{H^{s+2}(\Gamma)} \leq C\left( \|\psi_1\|_{H^s(\Omega)} + \|\psi_2\|_{H^s(\Gamma)} \right).
\end{equation}
\end{lemma}

The following result is the so-called transitivity property of exponential attraction from \cite[Theorem 5.1]{FGMZ04}).

\begin{proposition}  \label{t:exp-attr}
Let $(\mathcal{X},d)$ be a metric space and let $S_t$ be a semigroup acting on this space such that 
\[
d(S_t x_1,S_t x_2) \leq C e^{Kt} d(x_1,x_2),
\]
for appropriate constants $C$ and $K$. 
Assume that there exists three subsets $U_1$,$U_2$,$U_3\subset\mathcal{X}$ such that 
\[
{\rm{dist}}_\mathcal{X}(S_t U_1,U_2) \leq C_1 e^{-\alpha_1 t}, \quad{\rm{dist}}_\mathcal{X}(S_t U_2,U_3) \leq C_2 e^{-\alpha_2 t}.
\]
Then 
\[
{\rm{dist}}_\mathcal{X}(S_t U_1,U_3) \leq C' e^{-\alpha' t},
\]
where $C'=CC_1+C_2$ and $\alpha'=\frac{\alpha_1\alpha_1}{K+\alpha_1+\alpha_2}$.
\end{proposition}

The following statement refers to a frequently used Gr\"{o}nwall-type inequality that is useful when working with dissipation arguments. 
We also refer the reader to \cite[Lemma 2.1]{Conti-Pata-2005}, \cite[Lemma 2.2]{Grasselli&Pata02}, \cite[Lemma 5]{Pata&Zelik06}.

\begin{proposition}  \label{GL}
Let $\Lambda :\mathbb{R}_{+}\rightarrow \mathbb{R}_{+}$ be an absolutely continuous function satisfying 
\begin{equation*}
\frac{d}{dt}\Lambda (t)+2\eta \Lambda (t)\leq h(t)\Lambda(t)+k,
\end{equation*}
where $\eta >0$, $k\geq 0$ and $\int_{s}^{t}h(\tau )d\tau \leq \eta(t-s)+m$, for all $t\geq s\geq 0$ and some $m\geq 0$. Then, for all $t\geq 0$, 
\begin{equation*}
\Lambda (t)\leq \Lambda (0)e^{m}e^{-\eta t}+\frac{ke^{m}}{\eta }.
\end{equation*}
\end{proposition}

%%%%%%%%%%%%%%%%%%%%%%%%%%%%%%%%%%%%%%%%%
%%%%%%%%%%%%%%%%%%%%%%%%%%%%%%%%%%%%%%%%%
\section*{Acknowledgments}
%%%%%%%%%%%%%%%%%%%%%%%%%%%%%%%%%%%%%%%%%
%%%%%%%%%%%%%%%%%%%%%%%%%%%%%%%%%%%%%%%%%

The author is indebted to the anonymous referees for their careful reading of the manuscript and for their helpful comments and suggestions.

%===========================================================
%\bibliographystyle{amsplain}
%\bibliography{/Users/josephshomberg/Dropbox/LaTeX/REFERENCES}

\begin{thebibliography}{10}

\bibitem{Babin&Vishik92}
A.~V. Babin and M.~I. Vishik, \emph{Attractors of evolution equations},
  North-Holland, Amsterdam, 1992.

\bibitem{CGGM10}
Cecilia Cavaterra, Ciprian~G. Gal, Maurizio Grasselli, and Alain Miranville,
  \emph{Phase-field systems with nonlinear coupling and dynamic boundary
  conditions}, Nonlinear Anal. \textbf{72} (2010), no.~5, 2375--2399.

\bibitem{CDGP-2010}
Micka\"{e}l~D. Chekroun, Francesco~Di Plinio, Nathan~E. Glatt-Holtz, and
  Vittorino Pata, \emph{Asymptotics of the {C}oleman--{G}urtin model}, Discrete
  Contin. Dyn. Syst. Ser. S \textbf{4} (2011), no.~2, 351--369.

\bibitem{CFGGGOR09}
G.~M. Coclite, A.~Favini, C.~G. Gal, G.~R. Goldstein, J.~A. Goldstein,
  E.~Obrecht, and S.~Romanelli, \emph{The role of {W}entzell boundary
  conditions in linear and nonlinear analysis}, In: S. Sivasundaran. Advances
  in Nonlinear Analysis: Theory, Methods and Applications. vol 3, Cambridge
  Scientific Publishers Ltd., Cambridge, 2009.

\bibitem{Conti-Pata-2005}
Monica Conti and Vittorino Pata, \emph{Weakly dissipative semilinear equations
  of viscoelasticity}, Commun. Pure Appl. Anal. \textbf{4} (2005), no.~4,
  705--720.

\bibitem{CPS05}
Monica Conti, Vittorino Pata, and Marco Squassina, \emph{Singular limit of
  dissipative hyperbolic equations with memory}, Discrete Contin. Dyn. Syst.
  suppl. (2005), 200--208.

\bibitem{CPS06}
\bysame, \emph{Singular limit of differential systems with memory}, Indiana
  Univ. Math. J. \textbf{55} (2006), no.~1, 169--215.

\bibitem{EFNT95}
A.~Eden, C.~Foias, B.~Nicolaenko, and R.~Temam, \emph{Exponential attractors
  for dissipative evolution equations}, Research in Applied Mathematics, John
  Wiley and Sons Inc., 1995.

\bibitem{EMZ00}
Messoud Efendiev, Alain Miranville, and Sergey Zelik, \emph{Exponential
  attractors for a nonlinear reaction-diffusion systems in $\mathbb{R}^3$}, C.
  R. Acad. Sci. Paris S\'{e}r. I Math. \textbf{330} (2000), no.~8, 713--718.

\bibitem{FGMZ04}
P.~Fabrie, C.~Galusinski, A.~Miranville, and S.~Zelik, \emph{Uniform
  exponential attractors for singularly perturbed damped wave equations},
  Discrete Contin. Dyn. Syst. \textbf{10} (2004), no.~2, 211--238.

\bibitem{Frigeri&ShombergXX}
Sergio Frigeri and Joseph~L. Shomberg, \emph{Attractors for damped semilinear
  wave equations with a {R}obin--acoustic boundary perturbation}, ArXiv
  e-prints http://adsabs.harvard.edu/abs/2015arXiv150301821F (2015).

\bibitem{Gal12-2}
Ciprian~G. Gal, \emph{On a class of degenerate parabolic equations with dynamic
  boundary conditions}, J. Differential Equations \textbf{253} (2012),
  126--166.

\bibitem{Gal12-1}
\bysame, \emph{Sharp estimates for the global attractor of scalar
  reaction-diffusion equations with a {W}entzell boundary condition}, J.
  Nonlinear Sci. \textbf{22} (2012), no.~1, 85--106.

\bibitem{Gal-15Z}
\bysame, \emph{The role of surface diffusion in dynamic boundary conditions:
  where do we stand?}, Milan J. Math. \textbf{to appear} (20XX).

\bibitem{Gal&Grasselli08}
Ciprian~G. Gal and Maurizio Grasselli, \emph{The non-isothermal {A}llen--{C}ahn
  equation with dynamic boundary conditions}, Discrete Contin. Dyn. Syst.
  \textbf{22} (2008), no.~4, 1009--1040.

\bibitem{GGM08}
Ciprian~G. Gal, Maurizio Grasselli, and Alain Miranville, \emph{Nonisothermal
  {A}llen--{C}ahn equations with coupled dynamic boundary conditions},
  Nonlinear phenomena with energy dissipation \textbf{29} (2008), 117--139.

\bibitem{Gal-Shomberg15-2}
Ciprian~G. Gal and Joseph~L. Shomberg, \emph{Coleman-{G}urtin type equations
  with dynamic boundary conditions}, Phys. D \textbf{292/293} (2015), 29--45.

\bibitem{Gal&Shomberg15}
\bysame, \emph{Hyperbolic relaxation of reaction diffusion equations with
  dynamic boundary conditions}, Quart. Appl. Math. \textbf{73} (2015), no.~1,
  93--129.

\bibitem{Gal&Warma10}
Ciprian~G. Gal and Mahamadi Warma, \emph{Well posedness and the global
  attractor of some quasi-linear parabolic equations with nonlinear dynamic
  boundary conditions}, Differential Integral Equations \textbf{23} (2010),
  no.~3-4, 327--358.

\bibitem{GGMP05}
S.~Gatti, M.~Grasselli, A.~Miranville, and V.~Pata, \emph{A construction of a
  robust family of exponential attractors}, Proc. Amer. Math. Soc. \textbf{134}
  (2006), no.~1, 117--127.

\bibitem{GMPZ10}
S.~Gatti, A.~Miranville, V.~Pata, and S.~Zelik, \emph{Continuous families of
  exponential attractors for singularly perturbed equations with memory}, Proc.
  Roy. Soc. Edinburgh Sect. A \textbf{140} (2010), 329--366.

\bibitem{GMS2010}
Gianni Gilardi, Alain Miranville, and Giulio Schimperna, \emph{On the
  {C}ahn--{H}illiard equation with irregular potentials and dynamic boundary
  conditions}, Commun. Pure Appl. Anal. \textbf{8} (2009), no.~3, 881--912.

\bibitem{GPM98}
Claudio Giorgi, Vittorino Pata, and Alfredo Marzocchi, \emph{Asymptotic
  behavior of a semilinear problem in heat conduction with memory}, NoDEA
  Nonlinear Differential Equations Appl. \textbf{5} (1998), no.~3, 333--354.

\bibitem{GPM00}
\bysame, \emph{Uniform attractors for a non-autonomous semilinear heat equation
  with memory}, Quart. Appl. Math. \textbf{58} (2000), no.~4, 661--683.

\bibitem{Gold06}
G.~R. Goldstein, \emph{Derivation and physical interpretation of general
  boundary conditions}, Adv. in Diff. Eqns. \textbf{11} (2006), 457--480.

\bibitem{Grasselli&Pata02}
Maurizio Grasselli and Vittorino Pata, \emph{On the damped semilinear wave
  equation with critical exponent},  (2002).

\bibitem{Grasselli&Pata02-2}
\bysame, \emph{Uniform attractors of nonautonomous dynamical systems with
  memory}, Progr. Nonlinear Differential Equations Appl. \textbf{50} (2002),
  155--178.

\bibitem{Hale&Raugel88}
J.~Hale and G.~Raugel, \emph{Upper semicontinuity of the attractor for a
  singularly perturbed hyperbolic equation}, J. Differential Equations
  \textbf{73} (1988), no.~2, 197--214.

\bibitem{Kostin98}
I.~N. Kostin, \emph{Rate of attraction to a non-hyperbolic attractor},
  Asymptotic Anal. \textbf{16} (1998), no.~3, 203--222.

\bibitem{Lions&Magenes72}
J.~L. Lions and E.~Magenes, \emph{Non-homogeneous boundary value problems and
  applications}, vol.~I, Springer-Verlag, Berlin, 1972.

\bibitem{Milani&Koksch05}
Albert~J. Milani and Norbert~J. Koksch, \emph{An introduction to semiflows},
  Monographs and Surveys in Pure and Applied Mathematics - Volume 134, Chapman
  \& Hall/CRC, Boca Raton, 2005.

\bibitem{MPZ07}
Alain Miranville, Vittorino Pata, and Sergey Zelik, \emph{Exponential
  attractors for singularly perturbed damped wave equations: A simple
  construction}, Asymptot. Anal. \textbf{53} (2007), 1--12.

\bibitem{Pata&Zelik06}
V.~Pata and S.~Zelik, \emph{Smooth attractors for strongly damped wave
  equations}, Nonlinearity \textbf{19} (2006), no.~7, 1495--1506.

\bibitem{Pata-Zucchi-2001}
Vittorino Pata and Adele Zucchi, \emph{Attractors for a damped hyperbolic
  equation with linear memory}, Adv. Math. Sci. Appl. \textbf{11} (2001),
  no.~2, 505--529.

\bibitem{Pazy83}
Amnon Pazy, \emph{Semigroups of linear operators and applications to partial
  differential equations}, Applied Mathematical Sciences - Volume 44,
  Springer-Verlag, New York, 1983.

\bibitem{Robinson01}
James~C. Robinson, \emph{Infinite--dimensional dynamical systems}, Cambridge
  Texts in Applied Mathematics, Cambridge University Press, Cambridge, 2001.

\bibitem{RBT01}
A.~Rodr\'{i}guez-Bernal and A.~Tajdine, \emph{Nonlinear balance for
  reaction-diffusion equations under nonlinear boundary conditions:
  dissipativity and blow-up}, J. Differential Equations \textbf{169} (2001),
  332--372.

\bibitem{Tanabe79}
Hiroki Tanabe, \emph{Equations of evolution}, Pitman, London, 1979.

\bibitem{Temam88}
Roger Temam, \emph{Infinite-dimensional dynamical systems in mechanics and
  physics}, Applied Mathematical Sciences - Volume 68, Springer-Verlag, New
  York, 1988.

\end{thebibliography}
\providecommand{\bysame}{\leavevmode\hbox to3em{\hrulefill}\thinspace}
\providecommand{\MR}{\relax\ifhmode\unskip\space\fi MR }
\providecommand{\MRhref}[2]{%
  \href{http://www.ams.org/mathscinet-getitem?mr=#1}{#2}
}
\providecommand{\href}[2]{#2}

\end{document}